\documentclass[a4paper,12pt]{article}
\usepackage{amsmath} \usepackage{amssymb}
\usepackage{eucal}%,mathscr}
\usepackage{geometry}
\usepackage[amsmath,thmmarks]{ntheorem}
\theoremstyle{change}
\theoremseparator{.}

\theorembodyfont{\normalfont\slshape}
\newtheorem{theorem}{Theorem}[section]
\newtheorem{corollary}[theorem]{Corollary}
\newtheorem{lemma}[theorem]{Lemma}
\newtheorem{proposition}[theorem]{Proposition}

\theorembodyfont{\normalfont}
\newtheorem{definition}[theorem]{Definition}
\newtheorem{remark}[theorem]{Remark}

\theoremstyle{nonumberplain}
\theoremsymbol{\ensuremath{\scriptstyle{\blacksquare}}}
\newtheorem{proof}{Proof}
\theoremstyle{empty}
\newtheorem{proofof}{}

\numberwithin{equation}{section}

\def\phi{{\varphi}}

\newcommand\C{{\mathbb C}}
\newcommand\E{{\mathbb E}}
\newcommand\N{{\mathbb N}}
\newcommand\R{{\mathbb R}}
\newcommand\Z{{\mathbb Z}}
\newcommand\V{{\mathbb V}}
\newcommand\CA{{\mathcal A}}
\newcommand\CF{{\mathcal F}}
\newcommand\CL{{\mathcal L}}
\newcommand\qand{{\quad\mbox{and}\quad}}
\newcommand\GUE{{\text{\sf GUE}}}
\newcommand\Cov{{\sf Cov}}

\newcommand\rd{{\rm d}}
\newcommand\6{\, {\rm d}}
\newcommand\ri{{\rm i}}
\newcommand\e{{\rm e}}
\newcommand\unit{{{\pmb 1}}}
\newcommand\tr{{\text{\rm tr}}}
\newcommand\Tr{{\text{\rm Tr}}}
\newcommand\Ext{{\sf Ext}}
\newcommand\Cred{C^*_{\rm red}}
\newcommand\im{{\sf Im}}
\newcommand\re{{\sf Re}}
\newcommand\supp{{\rm supp}}

\begin{document}

\title{Asymptotic Expansions for the \\ Gaussian Unitary Ensemble\date{}}

\author{\sc Uffe Haagerup and Steen Thorbj{\o}rnsen}

\maketitle

\begin{abstract}
Let $g\colon{\mathbb R}\to{\mathbb \C}$ be a $C^{\infty}$-function with
all derivatives bounded and let ${\rm tr}_n$ denote the normalized
trace on the $n\times n$ matrices. In the paper \cite{em} Ercolani and
McLaughlin established asymptotic expansions of the mean value
${\mathbb E}\{{\rm tr}_n(g(X_n))\}$ for a rather general class of
random matrices $X_n$, including the Gaussian Unitary Ensemble
(GUE). Using an analytical approach, we provide in the present paper an
alternative proof of this asymptotic expansion in the GUE
case. Specifically we derive for a GUE random matrix $X_n$ that
\[
\E\big\{{\rm tr}_n(g(X_n))\big\}=
\frac{1}{2\pi}\int_{-2}^2g(x)\sqrt{4-x^2}\,{\rm d}x
+\sum_{j=1}^k\frac{\alpha_j(g)}{n^{2j}}+ O(n^{-2k-2}),
\]
where $k$ is an arbitrary positive integer.
Considered as mappings of $g$, we determine the coefficients
$\alpha_j(g)$, $j\in{\mathbb N}$, as distributions (in the sense of
L.~Schwarts). We derive a similar asymptotic expansion for the
covariance ${\sf Cov}\{{\rm Tr}_n[f(X_n)],{\rm Tr}_n[g(X_n)]\}$, 
where $f$ is a function of the same kind as $g$, and 
${\rm Tr}_n=n{\rm tr}_n$.
Special focus is drawn to the case where
$g(x)=\frac{1}{\lambda-x}$ and $f(x)=\frac{1}{\mu-x}$ for
$\lambda,\mu$ in ${\mathbb C}\setminus{\mathbb R}$. In this case the
mean and covariance considered above correspond to, respectively, the
one- and two-dimensional Cauchy (or Stieltjes) transform of the
${\sf GUE}(n,\frac{1}{n})$. 
\end{abstract}

\section{Introduction}

Since the groundbreaking paper \cite{vo} by Voiculescu, the
asymptotics for families of large, independent GUE random matrices has
become an important tool in the theory of operator algebras. In the
paper \cite{ht5} it was established that if
$X_1^{(n)},\ldots,X_r^{(n)}$ are independent $\GUE(n,\frac{1}{n})$
random matrices (see Definition~\ref{def gue} below), then with
probability one we have for any polynomial $p$ in $r$ non-commuting 
variables that
\begin{equation}
\lim_{n\to\infty}\big\|p\big(X_1^{(n)},\ldots,X_r^{(n)}\big)\big\|=
\|p(x_1,\ldots,x_r)\|,
\label{eq0.1}
\end{equation}
where $\{x_1,\ldots,x_r\}$ is a free semi-circular family of
selfadjoint operators in a $C^*$-probability space $(\CA,\tau)$ (see
\cite{vdn} for definitions), and where $\|\cdot\|$ denotes the
operator norm. This result leads in particular 
to the fact that there are non-invertible elements in the extension
semi-group of the reduced $C^*$-algebra associated to the free group
on $r$ generators (see \cite{ht5}).

A key step in the proof of \eqref{eq0.1} was to establish precise
estimates of the expectation and variance
of $\tr_n[g(p(X_1^{(n)},\ldots,X_r^{(n)}))]$, where $\tr_n$ denotes
the normalized trace, $g$ is a $C^\infty$-function with compact
support, and where we assume now that $p$ is a selfadjoint polynomial.
In fact it was established in \cite{ht5} and \cite{hst} that in this
setup we have the estimates:
\begin{align}
\E\big\{\tr_n\big[g\big(p(X_1^{(n)},\ldots,X_r^{(n)})\big)\big]\big\}
&=\tau[g(p(x_1,\ldots,x_r))]+O(n^{-2}),
\label{eq0.2}
\\[.2cm]
\V\big\{\tr_n\big[g\big(p(X_1^{(n)},\ldots,X_r^{(n)})\big)\big]\big\}
&=O(n^{-2}).
\label{eq0.3}
\end{align}
Furthermore, if the derivative $g'$ vanishes on the spectrum of the
operator $p(x_1,\ldots,x_r)$, then we actually have that 
\[
\V\big\{\tr_n\big[g\big(p(X_1^{(n)},\ldots,X_r^{(n)})\big)\big]\big\}
=O(n^{-4}).
\]
If we assume instead that $g$ is a polynomial, then the left hand
sides of \eqref{eq0.2} and \eqref{eq0.3} may actually be expanded as
polynomials in $n^{-2}$. More precisely it was proved in \cite{th}
that for any function $w\colon\{1,2,\ldots,p\}\to\{1,2,\ldots,r\}$
we have that\footnote{When $r=1$, formula \eqref{eq0.4} corresponds to the
  Harer-Zagier recursion formulas (see \cite{ht1}).}
\begin{equation}
\E\big\{\tr_n\big[X^{(n)}_{w(1)}X^{(n)}_{w(2)}\cdots
X^{(n)}_{w(p)}\big]\big\}
=\sum_{\gamma\in T(w)}n^{-2\sigma(\gamma)},
\label{eq0.4}
\end{equation}
where $T(w)$ is a certain class of permutations of $\{1,2,\ldots,p\}$,
and $\sigma(\gamma)\in\N_0$ for all $\gamma$ in $T(w)$ (see \cite{th} or
\cite{mn} for details). 
It was established furthermore in \cite{mn} that for two functions 
$w\colon\{1,2,\ldots,p\}\to\{1,2,\ldots,r\}$ and 
$v\colon\{1,2,\ldots,q\}\to\{1,2,\ldots,r\}$ we have that
\begin{equation}
\E\big\{\tr_n\big[X^{(n)}_{w(1)}X^{(n)}_{w(2)}\cdots
X^{(n)}_{w(p)}\big]\tr_n\big[X^{(n)}_{v(1)}X^{(n)}_{v(2)}\cdots
X^{(n)}_{v(q)}\big]\big\}
=\sum_{\gamma\in T(w,v)}n^{-2\sigma(\gamma)}
\label{eq0.5}
\end{equation}
where now $T(w,v)$ is a certain class of permutations of
$\{1,2,\ldots,p+q\}$ and again $\sigma(\gamma)\in\N_0$ for all
$\gamma$ in $T(w,v)$ (see \cite{mn} for details).

In view of \eqref{eq0.4} and \eqref{eq0.5} it is natural to ask
whether the left hand sides of \eqref{eq0.2} and \eqref{eq0.3} may
in general be expanded as ``power series'' in $n^{-2}$, when $g$ is, say,
a compactly supported $C^{\infty}$-function. In the case $r=1$,
this question was answered affirmatively by Ercolani and McLaughlin
(see \cite[Theorem~1.4]{em}) for a more general class of random matrices
than the GUE. More precisely, Ercolani and McLaughlin established for a
single matrix $X_n$ (from the considered class of random matrices) and 
any $C^{\infty}$-function $g$ with at most polynomial growth the
existence of a sequence $(\alpha_j(g))_{j\in\N_0}$ of complex
numbers, such that for any $n$ in $\N$ and $k$ in $\N_0$,
\begin{equation}
\E\big\{\tr_n(g(X_n))\big\}=
\sum_{j=0}^k\frac{\alpha_j(g)}{n^{2j}}+ O(n^{-2k-2}).
\label{eq0.6}
\end{equation}
Their proof is rather involved and is based on
Riemann-Hilbert techniques developed by Deift, McLaughlin and
co-authors. In this paper we provide an alternative proof
for \eqref{eq0.6} in the case 
where $X_n$ is a $\GUE(n,\frac{1}{n})$ random matrix. For technical
ease, we only establish \eqref{eq0.6} for functions in the class
$C_b^{\infty}(\R)$ consisting of all $C^{\infty}$-functions
$g\colon\R\to\C$, such that all derivatives $g^{(k)}$, $k\in\N_0$, are
bounded on $\R$. However, all (relevant) results of the present paper can
easily be extended to all $C^{\infty}$-functions with at
most polynomial growth. For each $j$ in $\N$ we show that the
coefficient $\alpha_j(g)$ is explicitly given in the form:
\[
\alpha_j(g)=\frac{1}{2\pi}\int_{-2}^2[T^jg](x)\sqrt{4-x^2}\6x
\]
for a certain linear operator $T\colon C_b^{\infty}(\R)\to
C_b^{\infty}(\R)$ (see Theorem~\ref{eksistens af T} and
Corollary~\ref{asymptotisk udv.}), and we describe $\alpha_j$
explicitly as a distribution (in the sense of L.~Schwarts) in terms of
Chebychev polynomials (cf.\ Corollary~\ref{distributioner}).
The proof of \eqref{eq0.6} is based
on the fact, proved by G\"otze and Tikhomirov in \cite{gt}, 
that the spectral density $h_n$ of a $\GUE(n,\frac{1}{n})$
random matrix satisfies the following third order differential
equation:
\begin{equation}
\frac{1}{n^2}h_n'''(x) + (4-x^2)h_n'(x) + xh_n(x) = 0,
\quad (x\in\R).
\label{eq0.7}
\end{equation}
In the special case where $g(x)=\frac{1}{\lambda-x}$ for some non-real
complex number $\lambda$, the integral $\int_{\R}g(x)h_n(x)\6x$ is the
Cauchy (or Stieltjes) transform $G_n(\lambda)$ for the measure
$h_n(x)\6x$, and asymptotic expansions like \eqref{eq0.6} appeared
already in the paper \cite{aps} for a rather general class of random
matrices (including the GUE). In the GUE case, our analytical approach
leads to the following explicit expansion (see Section~\ref{exp af
  Cauchy transform}):
\begin{equation}
G_n(\lambda)
=\eta_0(\lambda)+\frac{\eta_1(\lambda)}{n^2}+\frac{\eta_2(\lambda)}{n^4}
+\cdots+\frac{\eta_k(\lambda)}{n^{2k}}+O(n^{-2k-2}),
\label{eq0.12}
\end{equation}
where
\[
\eta_0(\lambda)=\frac{\lambda}{2}-\frac{1}{2}(\lambda^2-4)^{1/2},
\qand
\eta_j(\lambda)=\sum_{r=2j}^{3j-1}C_{j,r}(\lambda^2-4)^{-r-1/2} \quad(j\in\N).
\]
The constants $C_{j,r}$, $2j\le r\le 3j-1$, appearing above are
positive numbers for which we provide recursion formulas (see
Proposition~\ref{eksplicit udtryk for eta}). 

As for the ``power series expansion'' of \eqref{eq0.3}, consider again for
each $n$ a single $\GUE(n,\frac{1}{n})$ random matrix.
For any functions $f,g$ from $C_b^\infty(\R)$, we establish in
Section~\ref{2-dim udv} the existence of a sequence 
$(\beta_j(f,g))_{j\in\N_0}$ of complex numbers, such that for any $k$
in $\N_0$ and $n$ in $\N$,
\begin{equation}
\Cov\big\{\Tr_n[f(X_n)],\Tr_n[g(X_n)]\big\}
=\sum_{j=0}^k\frac{\beta_j(f,g)}{n^{2j}}+O(n^{-2k-2}),
\label{eq0.9}
\end{equation}
where $\Tr_n$ denotes the un-normalized trace on $M_n(\C)$,
and where the covariance $\Cov[Y,Z]$ of two
complex valued square integrable random variables is defined by
\[
\Cov\{Y,Z\}=\E\big\{(Y-\E\{Y\})(Z-\E\{Z\})\big\}.
\]
The proof of \eqref{eq0.9} is based on the following formula,
essentially due to Pastur and Scherbina (see \cite{ps}):
\begin{equation}
\Cov\big\{\Tr_n[f(X_n)],\Tr_n[g(X_n)]\big\}=
\int_{\R^2}\Big(\frac{f(x)-f(y)}{x-y}\Big)
\Big(\frac{g(x)-g(y)}{x-y}\Big)\rho_n(x,y)\6x\6y,
\label{eq0.10}
\end{equation}
where the kernel $\rho_n$ is given by
\[
\rho_n(x,y)=\tfrac{n}{4}\big[\varphi_n(\textstyle{\sqrt{\tfrac{n}{2}}x})
\varphi_{n-1}(\sqrt{\tfrac{n}{2}}y)-\varphi_{n-1}(\sqrt{\tfrac{n}{2}}x)
\varphi_n(\sqrt{\tfrac{n}{2}}y)\big]^2,
\]
with $\varphi_n$ the $n$'th Hermite function (see formula \eqref{e3.1}
below). The essential step then is to establish the formula
(see Theorem~\ref{formel for rho-n}):
\begin{equation}
\rho_n(x,y)=\frac{1}{4}\big[\tilde{h}_n(x)\tilde{h}_n(y)-4h_n'(x)h_n'(y)
-\tfrac{1}{n^2}h_n''(x)h_n''(y)\big], \qquad((x,y)\in\R^2),
\label{eq0.11}
\end{equation}
where $\tilde{h}_n(x)=h_n(x)-xh_n'(x)$, and $h_n$ is as before the spectral
density of $\GUE(n,\frac{1}{n})$. Using \eqref{eq0.10}-\eqref{eq0.11}
and Fubini's Theorem, the expansion \eqref{eq0.9} may be derived
from \eqref{eq0.6}. 

In the particular case where
\[
f(x)=\frac{1}{\lambda-x}, \qand g(x)=\frac{1}{\mu-x}, \qquad(x\in\R)
\]
we obtain in Section~\ref{sec to-dim G} the following specific
expansion for the two-dimensional Cauchy-transform:
\begin{equation}
\Cov\big\{\Tr_n[(\lambda\unit-X_n)^{-1}],\Tr_n[(\mu\unit-X_n)^{-1}]\big\}
=\frac{1}{2(\lambda-\mu)^2}\sum_{j=0}^k\frac{\Gamma_j(\lambda,\mu)}{n^{2j}}
+O(n^{-2k-2}),
\label{eq0.13}
\end{equation}
where the coefficients $\Gamma_j(\lambda,\mu)$ are given explicitly in
terms of the functions $\eta_l$ appearing in \eqref{eq0.12} (see
Corollary~\ref{asymp exp for 2-dim G}). The leading term
$\frac{\Gamma_0(\lambda,\mu)}{2(\lambda-\mu)^2}$ may also be identified as
the integral
\[
\int_{\R^2}\Big(\frac{(\lambda-x)^{-1}
-(\lambda-y)^{-1}}{x-y}\Big)\Big(\frac{(\mu-x)^{-1}
-(\mu-y)^{-1}}{x-y}\Big)\rho(x,y)\6x\6y,
\]
where $\rho(x,y)\6x\6y$ is the weak limit of the measures
$\rho_n(x,y)\6x\6y$ as $n\to\infty$ (cf.\ \eqref{eq0.10}).
The limiting density $\rho$ is explicitly given by 
\begin{equation}
\rho(x,y)=\frac{1}{4\pi^2}\frac{4-xy}{\sqrt{4-x^2}\sqrt{4-y^2}}
1_{(-2,2)}(x)1_{(-2,2)}(y), \qquad(x,y\in\R),
\label{eq0.14}
\end{equation}
and we provide a proof of this fact at the end of Section~\ref{2-dim
  udv}. 

In the paper \cite{aps} the authors derive for a rather general class
of random matrices an expansion for the two-dimensional Cauchy
transform in the form: 
\[
\Cov\big\{\Tr_n[(\lambda\unit-X_n)^{-1}],\Tr_n[(\mu\unit-X_n)^{-1}]\big\}
=\sum_{j=0}^kd_{j,n}(\lambda,\mu)n^{-j}+o(n^{-k}),
\]
where the leading coefficient $d_{0,n}$ is given explicitly by
\begin{equation}
d_{0,n}(\lambda,\mu)=\frac{1}{2(\lambda-\mu)^2}
\Big(\frac{\lambda\mu-a^2}{\sqrt{\lambda^2-a^2}\sqrt{\mu^2-a^2}}-1\Big),
\label{eq0.9a}
\end{equation}
with $a$ the variance of the relevant limiting semi-circle
distribution. In the GUE set-up considered in the present paper,
$a=2$, and in this case it is easily checked that $d_{0,n}$ is identical
to the leading coefficient
$\frac{\Gamma_0(\lambda,\mu)}{2(\lambda-\mu)^2}$ in \eqref{eq0.13}. 

The density $\rho$ given by \eqref{eq0.14} has
previously appeared in the paper \cite{c-d}. There the author proves
that if $X_n$ is a $\GUE(n,\frac{1}{n})$ random matrix, then for any
polynomial $f$, such that $\int_{-2}^2f(t)\sqrt{4-t^2}\6t=0$, the
random variable $\Tr_n(f(X_n))$ 
converges, as $n\to\infty$, in distribution to the Gaussian
distribution $N(0,\sigma^2)$, where the limiting variance $\sigma^2$
is given by
\[
\sigma^2=\int_{[-2,2]\times[-2,2]}
\Big(\frac{f(x)-f(y)}{x-y}\Big)^2\rho(x,y)\6x\6y.
\]
The density $\rho$ has also been identified in the physics
literature as the (leading term for the) correlation function
of the formal level density for the GUE (see \cite{kkp} and references
therein).

In a forthcoming paper (under preparation) we establish results
similar to those obtained in the present paper for random matrices of
Wishart type.

\section{Auxiliary differential equations}\label{sec 2}

In this section we consider two differential equations, both of which
play a crucial role in the definition of the operator $T$ introduced in
Section~\ref{sec 1-dim asympt exp}. The former is a third order
differential equation for the spectral density of the GUE.
We start thus by reviewing the GUE and its spectral distribution.

Consider a random $n\times n$ matrix $Y=(y_{ij})_{1\le i,j\le n}$
defined on some probability space $(\Omega,\CF,P)$. The distribution
of $Y$ is then the probability measure $P_Y$ on the set $M_n(\C)$ of
$n\times n$-matrices (equipped with Borel-$\sigma$-algebra) given by
\[
P_Y(B)=P(Y\in B)
\]
for any Borel-subset $B$ of $M_n(\C)$.

Throughout the paper we focus exclusively on the {\it Gaussian Unitary Ensemble}
(GUE), which is the class of random matrices defined as follows:

\begin{definition}\label{def gue}
Let $n$ be a positive integer and $\sigma^2$ a positive real number.
By $\GUE(n,\sigma^2)$ we then denote the distribution of a random
$n\times n$ matrix $X=(x_{ij})_{1\le i,j\le n}$ (defined on some
probability space) satisfying the following four conditions: 

\begin{enumerate}

\item For any $i,j$ in $\{1,2,\ldots,n\}$, $x_{ij}=\overline{x_{ji}}$.

\item The random variables $x_{ij}$, $1\le i\le j\le n$, are
independent.

\item If $1\le i<j\le n$, then $\re(x_{ij}),\im(x_{ij})$ are i.i.d.\
  with distribution $N(0,\frac{1}{2}\sigma^2)$.

\item For any $i$ in $\{1,\ldots,n\}$, $x_{ii}$ is a real-valued
  random variable with distribution $N(0,\sigma^2)$.

\end{enumerate}
\end{definition}

We recall now the specific form of the spectral distribution of a
$\GUE$ random matrix.
Let $\varphi_0,\varphi_1,\varphi_2,\ldots,$ be the sequence of Hermite
functions, i.e.,
\begin{equation}
\varphi_k(t)=\big(2^kk!\sqrt{\pi}\big)^{-1/2}H_k(t)\exp(-t^2/2),
\label{e3.1}
\end{equation}
where $H_0,H_1,H_2,\ldots,$ is the sequence of Hermite polynomials,
i.e.,
\begin{equation}
H_k(t)=(-1)^k\exp(t^2)\Big[\frac{\rd^k}{\rd t^k}\exp(-t^2)\Big].
\label{e3.2}
\end{equation}
Recall then (see e.g.\ \cite[Corollary~1.6]{ht1}) that the spectral
distribution of a random matrix $X$ from $\GUE(n,\frac{1}{n})$ has
density
\begin{equation}
h_n(t)=\frac{1}{\sqrt{2n}}\sum_{k=1}^{n-1}\textstyle{
\varphi_k\big(\sqrt{\frac{n}{2}}t\big)^2},  
\label{e3.3}
\end{equation}
w.r.t.\ Lebesgue measure. More precisely,
\[
\E\big\{\tr_n(g(X))\big\}=\int_{\R}g(t)h_n(t) \6t,
\]
for any Borel function $g\colon\R\to\R$, for which the integral on the
right hand side is well-defined. 

G\"otze and Tikhomirov established the following third order
differential equation for $h_n$:

\begin{proposition}[\cite{gt}]\label{diff-ligning for h}
For each $n$ in $\N$, $h_n$ is a solution to the differential
equation:
\[
\frac{1}{n^2}h_n'''(t) + (4-t^2)h_n'(t) + th_n(t) = 0,
\quad (t\in\R).
\]
\end{proposition}

\begin{proof} See \cite[Lemma~2.1]{gt}.
\end{proof}

Via the differential equation in
Proposition~\ref{diff-ligning for h} and integration by parts, we are
lead (see the proof of Theorem~\ref{eksistens af T} below)
to consider the following differential equation:
\begin{equation}
(t^2-4)f'(t) + 3tf(t) = g(t)
\label{e3.4a}
\end{equation}
for suitable given $C^{\infty}$-functions $g$. 
The same differential equation was studied by G\"otze and Tikhomorov in
\cite[lemma 3.1]{gt2} for a different class of functions $g$ in
connection with their Stein's Method approach to Wigner's semicircle law.

\begin{proposition}\label{loesning af diff-ligning}
For any $C^{\infty}$-function $g\colon\R\to\C$, the differential
equation
\begin{equation}
(t^2-4)f'(t) + 3tf(t) = g(t), \quad (t\in\R),
\label{e3.6}
\end{equation}
has unique $C^{\infty}$-solutions on $(-\infty,2)$ and on
$(-2,\infty)$. Furthermore, there is a $C^{\infty}$-solution to \eqref{e3.6}
on all of $\R$, if and only if $g$ satisfies
\begin{equation}
\int_{-2}^2 g(t)\sqrt{4-t^2}\6t = 0.
\label{e3.6a}
\end{equation}
\end{proposition}

\begin{proof} We note first that by splitting $f$ and $g$ in their real and
imaginary parts, we may assume that they are both real-valued functions.

{\sf Uniqueness:} By linearity it suffices to prove uniqueness
in the case $g=0$, i.e., that $0$ is the only solution to the
homogeneous equation:
\begin{equation}
(t^2-4)f'(t) + 3tf(t) = 0
\label{e3.7}
\end{equation}
on $(-\infty,2)$ and on $(-2,\infty)$.
By standard methods we can solve \eqref{e3.7} on each of the
intervals $(-\infty,-2)$, $(-2,2)$ and $(2,\infty)$. We find thus
that any solution to \eqref{e3.7} must satisfy that
\[
f(t)=
\begin{cases}
c_1(t^2-4)^{-3/2}, &\textrm{if} \ t<-2, \\
c_2(4-t^2)^{-3/2}, &\textrm{if} \ t\in(-2,2), \\
c_3(t^2-4)^{-3/2}, &\textrm{if} \ t>2,
\end{cases}
\]
for suitable constants $c_1,c_2,c_3$ in $\R$. Since a solution to
\eqref{e3.7} on $(-\infty,2)$ is continuous at $t=-2$, it follows that
for such a solution we must have $c_1=c_2=0$. Similarly, $0$ is the
only solution to \eqref{e3.7} on $(-2,\infty)$.

{\sf Existence:} The existence part is divided into three steps.

{\sf Step I}. We start by finding the solution to \eqref{e3.6} on
$(-2,\infty)$. By standard methods, it follows that the solution to
\eqref{e3.6} on $(2,\infty)$ is given by
\[
f_c(t)=(t^2-4)^{-3/2}\int_2^t(s^2-4)^{1/2} g(s) \6s +
c(t^2-4)^{-3/2}, \quad (t\in(2,\infty), \ c\in\C),
\]
whereas the solution to \eqref{e3.6} on $(-2,2)$ is given by
\[
f_c(t)=(4-t^2)^{-3/2}\int_t^2(4-s^2)^{1/2} g(s) \6s +
c(4-t^2)^{-3/2}, \quad (t\in(-2,2), \ c\in\C).
\]
Now consider the function $f\colon(-2,\infty)\to\R$ given by
\begin{equation}
f(t)=
\begin{cases}
(4-t^2)^{-3/2}\int_t^2(4-s^2)^{1/2} g(s) \6s, &\textrm{if} \ 
t\in(-2,2), \\
\frac{1}{6}g(2), &\textrm{if} \ t=2, \\
(t^2-4)^{-3/2}\int_2^t(s^2-4)^{1/2} g(s) \6s &\textrm{if} \
t\in(2,\infty).
\end{cases}
\label{e3.8}
\end{equation}
We claim that $f$ is a $C^{\infty}$-function on $(-2,\infty)$. Once
this has been verified, $f$ is automatically a solution to
\eqref{e3.6} on all of $(-2,\infty)$ (by continuity at $t=2$). To see
that $f$ is a $C^{\infty}$-function on $(-2,\infty)$, it suffices to
show that $f$ is $C^{\infty}$ on $(0,\infty)$, and for this we use the
following change of variables:
\[
y=t^2-4, \quad \textrm{i.e.,} \quad t=\sqrt{y+4}, \quad (t>2, \ y>0).
\]
For $y$ in $(0,\infty)$, we have
\[
f\big(\sqrt{4+y}\big)=y^{-3/2}\int_2^{\sqrt{4+y}}(s^2-4)^{1/2} g(s) \6s
=y^{-3/2}\int_0^yu^{1/2}\cdot\frac{g(\sqrt{u+4})}{2\sqrt{u+4}} \6u.
\]
Using then the change of variables
\[
u=vy, \quad v\in[0,1],
\]
we find that
\[
f\big(\sqrt{4+y}\big)
=y^{-3/2}\int_0^1v^{1/2}y^{1/2}\cdot\frac{g(\sqrt{4+vy})}{2\sqrt{4+vy}}
y \6v
=\int_0^1v^{1/2}\cdot\frac{g(\sqrt{4+vy})}{2\sqrt{4+vy}} \6v,
\]
for any $y$ in $(0,\infty)$. Now, consider the function
\[
l(y)=\int_0^1v^{1/2}\cdot\frac{g(\sqrt{4+vy})}{2\sqrt{4+vy}} \6v,
\]
which is well-defined on $(-4,\infty)$. By the usual theorem on
differentiation under the integral sign (see
e.g.\ \cite[Theorem~11.5]{sc}), it follows that $l$ is a
$C^{\infty}$-function on $(-4+\epsilon,K)$, for any positive 
numbers $\epsilon$ and $K$ such that $0<\epsilon<K$. Hence $l$ is a
$C^{\infty}$-function on all of $(-4,\infty)$. Note also that
\[
l(0)=\textstyle{\frac{1}{6}}g(2) = f(2)
\]
Furthermore, by performing change of variables as above in the reversed
order, we find for any $y$ in $(-4,0)$ that
\[
l(y)=(-y)^{-3/2}\int_{\sqrt{4+y}}^2(4-s^2)^{1/2} g(s) \6s
=f\big(\sqrt{4+y}\big).
\]
Hence, we have established that
$f(\sqrt{4+y})=l(y)$ for any $y$ in $(-4,\infty)$. Since
$l$ is a $C^{\infty}$-function on $(-4,\infty)$, and since
$f(t)=l(t^2-4)$ for all $t$ in $(0,\infty)$, it follows that
$f\in C^{\infty}((0,\infty))$, as desired.

{\sf Step II.} Next, we find the solution to \eqref{e3.6} on
$(-\infty,2)$. For this, consider the differential equation:
\begin{equation}
(t^2-4)\psi'(t) + 3t\psi(t) = g(-t), \quad (t\in(-2,\infty)).
\label{e3.9}
\end{equation}

From what we established in Step~I, it follows that \eqref{e3.9} has a
unique solution $\psi$ in $C^{\infty}((-2,\infty))$. Then put
\begin{equation}
f_1(t)=-\psi(-t), \quad (t\in(-\infty,2)),
\label{e3.9a}
\end{equation}
and note that $f_1\in C^{\infty}((-\infty,2))$, 
which satisfies \eqref{e3.6} on $(-\infty,2)$.

{\sf Step III.} It remains to verify that the solutions $f$ and $f_1$,
found in Steps I and II above, coincide on $(-2,2)$, if and only if
equation \eqref{e3.6a} holds. With $\psi$ as in
Step II, note that $\psi$ is given by the right hand side of
\eqref{e3.8}, if $g(s)$ is replaced by $g(-s)$. Thus, for any $t$ in
$(-2,2)$, we have that
\begin{equation*}\begin{split}
f(t)-f_1(t)&=f(t)+\psi(-t)\\[.2cm]
&=(4-t^2)^{-3/2}\int_t^2(4-s^2)^{1/2} g(s) \6s
+(4-t^2)^{-3/2}\int_{-t}^2(4-s^2)^{1/2} g(-s) \6s
\\[.2cm]
&=(4-t^2)^{-3/2}\int_t^2(4-s^2)^{1/2} g(s) \6s
+(4-t^2)^{-3/2}\int_{-2}^t(4-s^2)^{1/2} g(s) \6s\\[.2cm]
&=(4-t^2)^{-3/2}\int_{-2}^2(4-s^2)^{1/2} g(s) \6s,
\end{split}\end{equation*}
from which the assertion follow readily.
\end{proof}

\begin{proposition}\label{ligning II}
For any function $C^{\infty}$-function $g\colon\R\to\C$,
there is a unique $C^{\infty}$-function $f\colon\R\to\C$, such that
\begin{equation}
g(t)=\frac{1}{2\pi}\int_{-2}^2g(s)\sqrt{4-s^2}\6s+(t^2-4)f'(t)+3tf(t),
\qquad(t\in\R).
\label{e3.13}
\end{equation}
If $g\in C_b^{\infty}(\R)$, then $f\in C_b^{\infty}(\R)$ too.
\end{proposition}

\begin{proof} Let $g$ be a function from $C^{\infty}(\R)$, and consider the
function
\[
g_c=g-\frac{1}{2\pi}\int_{-2}^2g(s)\sqrt{4-s^2}\6s.
\]
Since $\int_{-2}^2 g_c(s)\sqrt{4-s^2}\6s=0$, it follows immediately
from Proposition~\ref{loesning af diff-ligning} that there is a
unique $C^{\infty}$-solution $f$ to \eqref{e3.13}. Moreover (cf.\ the proof of
Proposition~\ref{loesning af diff-ligning}), $f$ satisfies that
\[
f(t)=
\begin{cases}
(t^2-4)^{-3/2}\int_2^t(s^2-4)^{1/2} g_c(s) \6s, &\textrm{if} \
t\in(2,\infty), \\
-(t^2-4)^{-3/2}\int_2^{|t|}(s^2-4)^{1/2} g_c(-s) \6s, &\textrm{if} \
t\in(-\infty,-2).
\end{cases}
\]
Assume now that $g$ (and hence $g_c$) is in $C_b^{\infty}(\R)$, 
and choose a number $R$ in
$(0,\infty)$, such that $|g_c(t)|\le R$ for all $t$ in $\R$. Then, for any
$t$ in $(2,\infty)$, we find that
\[
|f(t)|\le (t^2-4)^{-3/2}R\int_2^t (s^2)^{1/2} \6s =
 \textstyle{\frac{1}{2}}R(t^2-4)^{-1/2},
\]
and thus $f$ is bounded on, say, $(3,\infty)$. It follows similarly that
$f$ is bounded on, say, $(-\infty,-3)$. Hence, since $f$ is continuous, $f$
is bounded on all of $\R$.

Taking first derivatives in \eqref{e3.13}, we note next that
\[
(t^2-4)f''(t)+5tf'(t) + 3f(t) = g'(t), \quad (t\in\R),
\]
and by induction we find that in general
\[
(t^2-4)f^{(k+1)}(t)+(2k+3)tf^{(k)}(t)+k(k+2)f^{(k-1)}(t) = g^{(k)}(t), 
\quad (k\in\N, \ t\in\R).
\]
Thus, for $t$ in $\R\setminus\{-2,2\}$,
\begin{equation}
f'(t)=\frac{-3tf(t)}{t^2-4}+\frac{g_c(t)}{t^2-4},
\label{e3.11}
\end{equation}
and
\begin{equation}
f^{(k+1)}(t)=\frac{-(2k+3)tf^{(k)}(t)}{t^2-4}
-\frac{k(k+2)f^{(k-1)}(t)}{t^2-4} +\frac{g^{(k)}(t)}{t^2-4}, \quad
(k\in\N).
\label{e3.12}
\end{equation}
Since $f$ and $g_c$ are bounded, it follows from \eqref{e3.11} that
$f'$ is bounded on, say, $\R\setminus[-3,3]$ and hence on all of
$\R$. Continuing by induction,
it follows similarly from \eqref{e3.12} that $f^{(k)}$ is bounded for all
$k$ in $\N$.
\end{proof}

\section{Asymptotic expansion for expectations of traces} 
\label{sec 1-dim asympt exp}

In this section we establish the asymptotic expansion
\eqref{eq0.6}. We start by introducing the vector space of infinitely
often bounded differentiable functions.

\begin{definition}\label{def frechet rum}
By $C_b^{\infty}(\R)$ we denote the vector space of
$C^{\infty}$-functions $f\colon\R\to\C$, satisfying that
\[
\forall k\in\N_0\colon \sup_{t\in\R}\Big| \frac{\rd^k}{\rd t^k}f(t)\Big| <
\infty.
\]
Moreover, we introduce a sequence $\|\cdot\|_{(k)}$ of norms on
$C_b^{\infty}(\R)$ as follows:
\[
\|g\|_{\infty}=\sup_{x\in\R}|g(x)|, \qquad(g\in C_b^{\infty}(\R)),
\]
and for any $k$ in $\N_0$:
\[
\|g\|_{(k)}=\max_{j=0,\ldots,k}\|g^{(j)}\|_{\infty},
\qquad(g\in C_b^{\infty}(\R)),
\]
where $g^{(j)}$ denotes the $j$'th derivative of $g$.
Equipped with the sequence $(\|\cdot\|_{(k)})_{k\in\N}$ of norms,
$C_b^{\infty}(\R)$ becomes a Fr\'ech\'et space (see e.g.\
Theorem~1.37 and Remark~1.38(c) in \cite{ru}).
\end{definition}

The following lemma is well-known, but for the reader's convenience we
include a proof.

\begin{lemma}\label{kont i Frechet rum}
Consider $C_b^{\infty}(\R)$ as a F\'ech\'et space as described in
Definition~\ref{def frechet rum}. Then
a linear mapping $L\colon C_b^{\infty}(\R)\to C_b^{\infty}(\R)$ is
continuous, if and only if the following condition is satisfied:
\begin{equation}
\forall k\in\N \ \exists m_k\in\N \ \exists C_k>0 \ \forall g\in
C_b^{\infty}(\R)\colon
\|Lg\|_{(k)}\le C_k\|g\|_{(m_k)}.
\label{e3.13a}
\end{equation}
\end{lemma}

\begin{proof} A sequence $(g_n)$ from $C_b^{\infty}(\R)$ converges to a
function $g$ in $C_b^{\infty}(\R)$ in the described Fr\'ech\'et
topology, if and only if $\|g_n-g\|_{(k)}\to0$ as $n\to\infty$ for any $k$
in $\N$. Therefore condition \eqref{e3.13a} clearly implies continuity
of $L$.

To establish the converse implication, 
note that by \cite[Theorem~1.37]{ru}, a neighborhood basis at $0$ for
$C_b^{\infty}(\R)$ is given by
\[
U_{k,\epsilon}=\{h\in C_b^{\infty}(\R)\mid \|h\|_{(k)}<\epsilon\},
\qquad(k\in\N, \ \epsilon>0).
\]
Thus, if $L$ is continuous, there exists for any $k$ in $\N$
an $m$ in $\N$ and a positive $\delta$, such that
$L(U_{m,\delta})\subseteq U_{k,1}$.
For any non-zero function $g$ in $C_b^{\infty}(\R)$, we have that
$\frac{1}{2}\delta\|g\|_{(m)}^{-1}g\in U_{m,\delta}$, and therefore
\[
\big\|\tfrac{1}{2}\delta\|g\|_{(m)}^{-1}Lg\big\|_{(k)}<1, 
\quad\mbox{i.e.\ }\quad
\|Lg\|_{(k)}<\tfrac{2}{\delta}\|g\|_{(m)},
\]
which establishes \eqref{e3.13a}.
\end{proof}

\begin{remark}\label{def af S og T}
Appealing to Proposition~\ref{ligning II}, we may define a mapping
$S\colon C_b^{\infty}(\R)\to C_b^{\infty}(\R)$ by setting, for $g$ in
$C_b^{\infty}(\R)$, $Sg=f$,
where $f$ is the unique solution to \eqref{e3.13}. By uniqueness, $S$
is automatically a linear mapping.
We define next the linear mapping $T\colon C_b^{\infty}(\R)\to
C_b^{\infty}(\R)$ by the formula:
\[
Tg=(Sg)''', \qquad(g\in C_b^{\infty}(\R)).
\]
\end{remark}

\begin{proposition}\label{kontinuitet af S}
The linear mappings $S,T\colon C_b^{\infty}(\R)\to C_b^{\infty}(\R)$ 
introduced in Remark~\ref{def af S og T} are
continuous when $C_b^{\infty}(\R)$ is viewed as a Fr\'ech\'et space as
described in Definition~\ref{def frechet rum}. 
\end{proposition}

\begin{proof} Since differentiation is clearly a continuous mapping from
$C_b^{\infty}(\R)$ into itself, it follows immediately that $T$ is
continuous, if $S$ is.

To prove that $S$ is continuous, it suffices to show that the
graph of $S$ is closed in 
$C_b^{\infty}(\R)\times C_b^{\infty}(\R)$ equipped with the product topology
(cf.\ \cite[Theorem~2.15]{ru}). So let 
$(g_n)$ be a sequence of functions in $C_b^{\infty}(\R)$, such that 
$(g_n,Sg_n)\to(g,f)$ in $C_b^{\infty}(\R)\times C_b^{\infty}(\R)$ for
some functions $f,g$ in $C_b^{\infty}(\R)$. In particular then,
\[
g_n\to g, \quad Sg_n\to f, \qand (Sg_n)'\to f' \quad\mbox{uniformly
  on $\R$ as $n\to\infty$}.
\]
It follows that for any $t$ in $\R$,
\begin{equation*}\begin{split}
g(t)&=\lim_{n\to\infty}g_n(t)
=\lim_{n\to\infty}\Big(\frac{1}{2\pi}\int_{-2}^2g_n(s)\sqrt{4-s^2}\6s
+(t^2-4)(Sg_n)'(t)+3t(Sg_n)(t)\Big)\\[.2cm]
&=\frac{1}{2\pi}\int_{-2}^2g(s)\sqrt{4-s^2}\6s
+(t^2-4)f'(t)+3tf(t).
\end{split}\end{equation*}
Therefore, by uniqueness of solutions to \eqref{e3.13}, $Sg=f$, and the
graph of $S$ is closed.
\end{proof}

\begin{theorem}\label{eksistens af T}
Consider the spectral density $h_n$ for $\GUE(n,\frac{1}{n})$ and the
linear operator $T\colon C_b^{\infty}(\R)\to C_b^{\infty}(\R)$
introduced in Remark~\ref{def af S og T}. Then for any function $g$ in
$C_b^{\infty}(\R)$ we have that
\[
\int_{\R} g(t) h_n(t) \6t = \frac{1}{2\pi}\int_{-2}^2g(t)\sqrt{4-t^2}
\6t + \frac{1}{n^2}\int_{\R}Tg(t)\cdot h_n(t) \6t.
\]
\end{theorem}

\begin{proof} Consider a fixed function $g$ from $C_b^{\infty}(\R)$, and then
put $f=Sg$, where $S$ is the linear operator introduced in Remark~\ref{def
  af S og T}. Recall that
\begin{equation}
g(t)=\frac{1}{2\pi}\int_{-2}^2g(s)\sqrt{4-s^2}\6s+(t^2-4)f'(t)+3tf(t),
\qquad(t\in\R).
\label{e3.17}
\end{equation}
By Proposition~\ref{diff-ligning for h} and partial integration
it follows that
\begin{equation*}\begin{split}
0&=\int_{\R}f(t)\big[n^{-2}h_n'''(t) + (4-t^2)h_n'(t) +
th_n(t)\big] \6t \\[.2cm]
&=-n^{-2}\int_{\R}f'''(t)h_n(t) \6t -
\int_{\R}\frac{\rd}{\rd t}\big[f(t)(4-t^2)\big] h_n(t) \6t + \int_{\R}t f(t)
h_n(t) \6t \\[.2cm]
&=\int_{\R}\big[-n^{-2}f'''(t) - (4-t^2)f'(t) + 3tf(t)\big] h_n(t) \6t, 
\end{split}\end{equation*}
so that
\[
\int_{\R}\big[(t^2-4)f'(t) + 3tf(t)\big] h_n(t) \6t =
\frac{1}{n^2}\int_{\R}f'''(t)h_n(t) \6t
=\frac{1}{n^2}\int_{\R}Tg(t)\cdot h_n(t) \6t.
\]
Using \eqref{e3.17} and the fact that $h_n(t)\6t$ is a probability
measure, we conclude that
\begin{equation*}\begin{split}
\int_{\R}g(t)\cdot h_n(t)\6t
&=\frac{1}{2\pi}\int_{-2}^2g(t)\sqrt{4-t^2}\6t
+\int_{\R}\big[(t^2-4)f'(t) + 3tf(t)\big] h_n(t) \6t\\[.2cm]
&=\frac{1}{2\pi}\int_{-2}^2g(t)\sqrt{4-t^2}\6t
+\frac{1}{n^2}\int_{\R}Tg(t)\cdot h_n(t) \6t,
\end{split}\end{equation*}
which is the desired expression.
\end{proof}

As an easy corollary of Proposition~\ref{eksistens af T}, we may now
derive (in the GUE case) Ercolani's and McLaughlin's asymptotic
expansion \cite[Theorem~1.4]{em}.

\begin{corollary}\label{asymptotisk udv.}  
Let $T\colon C_b^{\infty}(\R)\to C_b^{\infty}(\R)$ be
the linear mapping introduced in Remark~\ref{def af S og T}.
Then for any $k$ in $\N$ and $g$ in $C_b^{\infty}(\R)$, we have:
\begin{equation*}\begin{split}
\int_{\R} g(t) h_n(t) \6t &=
\frac{1}{2\pi}\sum_{j=0}^{k-1}\frac{1}{n^{2j}}
\int_{-2}^2[T^jg](t)\sqrt{4-t^2} \6t + 
\frac{1}{n^{2k}}\int_{\R}[T^kg](t)\cdot h_n(t) \6t.
\\[.2cm]
&=\frac{1}{2\pi}\sum_{j=0}^{k-1}\frac{1}{n^{2j}}
\int_{-2}^2[T^jg](t)\sqrt{4-t^2} \6t + O(n^{-2k}).
\end{split}\end{equation*}
\end{corollary}

\begin{proof} The first equality in the corollary follows immediately by
successive applications of Theorem~\ref{eksistens af T}.
To show the second one, it remains to establish that for any $k$ in $\N$
\[
\sup_{n\in\N}\int_{\R}\big|[T^kg](t)\big|\cdot h_n(t)\6t<\infty.
\]
But this follows immediately from the fact that $T^kg$ is bounded,
and the fact that $h_n(t)\6t$ is a probability measure for each $n$.
\end{proof}

\section{Asymptotic expansion for the Cauchy transform}
\label{exp af Cauchy transform}

For a $\GUE(n,\frac{1}{n})$ random matrix $X_n$, we consider now
the Cauchy transform given by 
\[
G_n(\lambda)=\E\big\{\tr_n\big[(\lambda\unit_n-X_n)^{-1}\big]\big\}
=\int_{\R}\frac{1}{\lambda-t}h_n(t)\6t,
\qquad(\lambda\in\C\setminus\R).
\]
Setting
\[
g_\lambda(t)=g(\lambda,t)=\frac{1}{\lambda-t}, \qquad(t\in\R, \
\lambda\in\C\setminus\R),
\]
we have by the usual theorem on differentiation under the integral
sign (for analytical functions) that $G_n$ is analytical on
$\C\setminus\R$ with derivatives
\begin{equation}
\frac{\rd^k}{\rd \lambda^k}G_n(\lambda)
=\int_{\R}\frac{(-1)^kk!}{(\lambda-t)^{k+1}}h_n(t)\6t
=(-1)^k\int_{\R}\Big(\frac{\rd^k}{\rd t^k}g_\lambda(t)\Big)h_n(t)\6t
\label{e3.15}
\end{equation}
for any $k$ in $\N$ and $\lambda$ in $\C\setminus\R$.

\begin{lemma}\label{diff ligning for G-n}
The Cauchy transform $G_n$ of a $\GUE(n,\frac{1}{n})$ random matrix
$X_n$ satisfies the following differential equation:
\begin{equation}
n^{-2}\frac{\rd^3}{\rd\lambda^3}G_n(\lambda)
+(4-\lambda^2)\frac{\rd}{\6\lambda}G_n(\lambda)
+\lambda G_n(\lambda)=2,
\label{e3.14}
\end{equation}
for all $\lambda$ in $\C\setminus\R$.
\end{lemma}

\begin{proof}
From Proposition~\ref{diff-ligning for h} and partial integration we
obtain for fixed $\lambda$ in $\C\setminus\R$ that
\begin{equation}
\begin{split}
0&=\int_{\R}g_\lambda(t)
\big[n^{-2}h_n'''(t) + (4-t^2)h_n'(t) + th_n(t)\big]\6t
\\[.2cm]
&=\int_{\R}\big[-n^{-2}g_\lambda'''(t)
-(4-t^2)g_\lambda'(t)+3tg_\lambda(t)\big]h_n(t)\6t.
\end{split}
\label{e3.16}
\end{equation}
Note here that
\[
(4-t^2)g_\lambda'(t)=\frac{4-t^2}{(\lambda-t)^2}
=\frac{4-\lambda^2}{(\lambda-t)^2}+\frac{2\lambda}{\lambda-t}-1,
\]
and that
\[
3tg_\lambda(t)=\frac{3t}{\lambda-t}=\frac{3\lambda}{\lambda-t}-3.
\]
Inserting this into \eqref{e3.16} and using \eqref{e3.15} and the fact
that $h_n$ is a probability density, we find that
\[
0=n^{-2}\frac{\rd^3}{\rd \lambda^3}G_n(\lambda)+
(4-\lambda^2)\frac{\rd}{\rd \lambda}G_n(\lambda)+\lambda G_n(\lambda)-2,
\]
as desired.
\end{proof}

For each fixed $\lambda$ in $\C\setminus\R$, we apply next
Corollary~\ref{asymptotisk udv.} to the function $g_\lambda$ and
obtain for any $k$ in $\N_0$ the expansion:
\begin{equation}
G_n(\lambda)=\int_{\R}\frac{1}{\lambda-t}h_n(t)\6t
=\eta_0(\lambda)+\frac{\eta_1(\lambda)}{n^2}+\frac{\eta_2(\lambda)}{n^4}+\cdots
+\frac{\eta_k(\lambda)}{n^{2k}}+O(n^{-2k-2}),
\label{e3.18}
\end{equation}
where
$\eta_j(\lambda)=\frac{1}{2\pi}\int_{-2}^2[T^jg_\lambda](t)\sqrt{4-t^2}\6t$
for all $j$. To determine these coefficients we shall insert the
expansion \eqref{e3.18} into the differential equation \eqref{e3.14}
in order to obtain differential equations for the $\eta_j$'s. To make
this rigorous, we need first to establish analyticity of the
$\eta_j$'s as functions of $\lambda$.

\begin{lemma}\label{analytisk}

\begin{itemize}

\item[(i)] For any $k$ in $\N_0$ the mapping $\lambda\mapsto
  T^kg_\lambda$ is analytical as a mapping from $\C\setminus\R$ into
  the Fr\'ech\'et space $C_b^{\infty}(\R)$, and
\[
\frac{\rd^j}{\rd\lambda^j}T^kg_\lambda
=T^k\Big(\frac{\partial^j}{\partial\lambda^j}g(\lambda,\cdot)\Big) 
\quad\mbox{for any $j$ in $\N$.}
\]

\item[(ii)] For any $k,n$ in $\N$, consider the mappings
  $\eta_k,R_{k,n}\colon\C\setminus\R\to\C$ given by
\begin{align}
\eta_k(\lambda)&=\int_{-2}^2[T^kg_\lambda](s)\sqrt{4-s^2}\6s,
\qquad(\lambda\in\C\setminus\R)
\label{e3.18a}\\[.2cm]
R_{k,n}(\lambda)&=\int_{\R}[T^{k+1}g_\lambda](s)h_n(s)\6s,
\qquad(\lambda\in\C\setminus\R).\label{e3.18d}
\end{align}
These mappings are analytical on $\C\setminus\R$ with derivatives:
\begin{align*}
\frac{\rd^j}{\rd\lambda^j}\eta_k(\lambda)
&=\int_{-2}^2\big[T^k\big(\tfrac{\partial^j}{\partial\lambda^j}
g(\lambda,\cdot)\big)\big](s)\sqrt{4-s^2}\6s,
\qquad(\lambda\in\C\setminus\R, \ j\in\N)\\[.2cm]
\frac{\rd^j}{\rd\lambda^j}R_{k,n}(\lambda)&=\int_{\R}
\big[T^{k+1}\big(\tfrac{\partial^j}{\partial\lambda^j}
g(\lambda,\cdot)\big)\big](s)h_n(s)\6s,
\qquad(\lambda\in\C\setminus\R, \ j\in\N).
\end{align*}
\end{itemize}
\end{lemma}

\begin{proof} (i) \ By standard methods it follows that for any $\lambda$ in
$\C\setminus\R$ and $l,j$ in $\N_0$, 
\begin{equation}
\lim_{h\to0}\Big(\sup_{t\in\R}\Big|\frac{1}{h}
\Big(\frac{\partial^l}{\partial t^l}\frac{\partial^j}{\partial\lambda^j}
g(\lambda+h,t)
-\frac{\partial^l}{\partial t^l}\frac{\partial^j}{\partial\lambda^j}
g(\lambda,t)\Big)
-\frac{\partial^{l}}{\partial t^{l}}
\frac{\partial^{j+1}}{\partial\lambda^{j+1}}g(\lambda,t)\Big|\Big)
=0.
\label{e3.18c}
\end{equation}
When $j=0$, formula \eqref{e3.18c} shows that the mapping
$F\colon\C\setminus\R\to C_b^{\infty}(\R)$ given by
\[
F(\lambda)=g(\lambda,\cdot), \quad(\lambda\in\C\setminus\R),
\]
is analytical on $\C\setminus\R$ with derivative
$\frac{\rd}{\rd\lambda}F(\lambda)
=\frac{\partial}{\partial\lambda}g(\lambda,\cdot)$
(cf.\ \cite[Definition~3.30]{ru}). Using then \eqref{e3.18c} and
induction on $j$, it follows that moreover
\[
\frac{\rd^j}{\rd\lambda^j}F(\lambda)
=\frac{\partial^j}{\partial\lambda^j}g(\lambda,\cdot), 
\qquad(\lambda\in\C\setminus\R)
\]
for all $j$ in $\N$.
For each $k$ in $\N$ the mapping $T^k\colon C_b^{\infty}(\R)\to
C_b^{\infty}(\R)$ is linear and continuous (cf.\
Proposition~\ref{kontinuitet af S}), and it follows therefore
immediately that the composed mapping $T^k\circ F\colon\C\setminus\R\to
C_b^{\infty}(\R)$ is again analytical on $\C\setminus\R$ with
derivatives 
\[
\frac{\rd^j}{\rd\lambda^j}T^kg(\lambda,\cdot)
=\frac{\rd^j}{\rd\lambda^j}T^k\circ F(\lambda)
=T^k\Big(\frac{\partial^j}{\partial\lambda^j}g(\lambda,\cdot)\Big)
\quad\mbox{for all $j$ in $\N$.}
\]
This establishes (i).

(ii) \ As an immediate consequence of (i), for each fixed $s$ in $\R$
the mapping
$\lambda\mapsto [T^kg(\lambda,\cdot)](s)$ is analytical with derivatives
\[
\frac{\rd^j}{\rd\lambda^j}[T^kg(\lambda,\cdot)](s)=
\Big[T^k\frac{\partial^j}{\partial\lambda^j}
g(\lambda,\cdot)\Big](s), \qquad(j\in\N).
\]
Note here that by Lemma~\ref{kont i Frechet rum}
\[
\Big\|T^k\frac{\partial^j}{\partial\lambda^j}g(\lambda,\cdot)\Big\|_{\infty}\le
C(k,0)\Big\|\frac{\partial^j}{\partial\lambda^j}g(\lambda,\cdot)\Big\|_{(m(k,0))}
\]
for suitable constants $C(k,0)$ in $(0,\infty)$ and $m(k,0)$ in 
$\N$. Hence, for any closed ball $B$ inside $\C\setminus\R$ and any
$j$ in $\N$ we have that
\[
\sup_{\lambda\in B}\Big\|T^k\Big(\frac{\partial^j}{\partial\lambda^j}
g(\lambda,\cdot)\Big)\Big\|_{\infty}
\le C(k,0)\sup_{\lambda\in B}
\Big\|\frac{\partial^j}{\partial\lambda^j}
g(\lambda,\cdot)\Big\|_{(m(k,0))}<\infty.
\]
It follows now by application of the usual theorem on differentiation
under the integral sign, that for any finite
Borel-measure $\mu$ on $\R$, the mapping
$\lambda\mapsto\int_{\R}[T^kg(\lambda,\cdot)](s)\,\mu(\rd s)$ is
analytical on $\C\setminus\R$ with derivatives
\[
\frac{\rd^j}{\rd\lambda^j}\int_{\R}[T^kg(\lambda,\cdot)](s)\,\mu(\rd s)
=\int_{\R}\frac{\rd^j}{\rd\lambda^j}[T^kg(\lambda,\cdot)](s)\,\mu(\rd s)
=\int_{\R}\big[T^k\big(\tfrac{\partial^j}{\partial\lambda^j}
g(\lambda,\cdot)\big)\big](s)\,\mu(\rd s).
\]
In particular this implies (ii).
\end{proof}

\begin{lemma}\label{diff lign for eta-j} 
Let $G_n$ denote the Cauchy-transform of $h_n(x)\6x$,
  and consider for each $\lambda$ in $\C\setminus\R$ and $k$ in
  $\N_0$ the asymptotic expansion:
\begin{equation}
G_n(\lambda)
=\eta_0(\lambda)+\frac{\eta_1(\lambda)}{n^2}+\frac{\eta_2(\lambda)}{n^4}+\cdots
+\frac{\eta_k(\lambda)}{n^{2k}}+O(n^{-2k-2})
\label{e3.19}
\end{equation}
given by Corollary~\ref{asymptotisk udv.}. Then
the coefficients $\eta_j(\lambda)$ are analytical as functions of
$\lambda$, and they satisfy the following recursive system of
differential equations:
\begin{align}
(4-\lambda^2)\eta_0'(\lambda)+\lambda\eta_0(\lambda)&= 2\notag\\[.2cm]
(\lambda^2-4)\eta_j'(\lambda)-\lambda
\eta_j(\lambda)&=\eta_{j-1}'''(\lambda)\label{e3.22},
\qquad (j\in\N).
\end{align}
\end{lemma}

\begin{proof} For each $j$ in $\N_0$ the coefficient $\eta_j(\lambda)$ is
given by \eqref{e3.18a} (cf.\ Corollary~\ref{asymptotisk udv.}), and
hence Lemma~\ref{analytisk} asserts that $\eta_j$ is analytical on
$\C\setminus\R$. Recall also from Corollary~\ref{asymptotisk udv.} that
the $O(n^{-2k-2})$ term in \eqref{e3.19} has the form
$n^{-2k-2}R_{k,n}(\lambda)$, where $R_{k,n}(\lambda)$ is given by
\eqref{e3.18d} and is again an analytical function on $\C\setminus\R$
according to Lemma~\ref{analytisk}. Inserting now \eqref{e3.19} into the
differential equation \eqref{e3.14}, we obtain for $\lambda$ in
$\C\setminus\R$ that   
\begin{equation} 
\begin{split}
2&=n^{-2}G_n'''(\lambda)+(4-\lambda^2)G_n'(\lambda)+\lambda
G_n(\lambda)\\[.2cm]
&=n^{-2}\Big(\sum_{j=0}^kn^{-2j}\eta_j'''(\lambda)
+n^{-2k-2}R_{k,n}'''(\lambda)\Big)
+(4-\lambda^2)\Big(\sum_{j=0}^kn^{-2j}\eta_j'(\lambda)
+n^{-2k-2}R_{k,n}'(\lambda)\Big)\\
&\ \phantom{=n^{-2}\Big(\sum_{j=0}^kn^{-2j}\eta_j'''(\lambda)
+n^{-2k-2}R_{k,n}'''(\lambda)\Big)}
+\lambda\Big(\sum_{j=0}^kn^{-2j}\eta_j(\lambda)
+n^{-2k-2}R_{k,n}(\lambda)\Big)
\\[.2cm]
&=\big[(4-\lambda^2)\eta_0'(\lambda)+\lambda\eta_0(\lambda)\big]
+\sum_{j=1}^kn^{-2j}\big[\eta_{j-1}'''(\lambda)
+(4-\lambda^2)\eta_j'(\lambda)+\lambda\eta_j(\lambda)\big]\\
&\quad\phantom{\big[(4-\lambda^2)\eta_0'(\lambda)\big]}
+n^{-2k-2}\big[\eta_k'''(\lambda)+(4-\lambda^2)R_{k,n}'(\lambda)
+\lambda R_{k,n}(\lambda)\big]
+n^{-2k-4}R_{k,n}'''(\lambda).
\end{split}
\label{e3.20}
\end{equation}
Using Lemma~\ref{analytisk}, we note here that for fixed $k$ and
$\lambda$ we have for any $l$ in $\N_0$ that
\[
\sup_{n\in\N}\big|\tfrac{\rd^l}{\rd\lambda^l}R_{k,n}(\lambda)\big|=
\sup_{n\in\N}\Big|\int_{\R}
\big[T^{k+1}\big(\tfrac{\partial^l}{\partial\lambda^l}
g(\lambda,\cdot)\big)\big](s)h_n(s)\6s\Big|
\le\big\|T^{k+1}\big(\tfrac{\partial^l}{\partial\lambda^l}
g(\lambda,\cdot)\big)\big\|_{\infty}
<\infty,
\]
since $T^{k+1}\big(\tfrac{\partial^l}{\partial\lambda^l}g(\lambda,\cdot)\big)
\in C_b^{\infty}(\R)$. Thus, letting $n\to\infty$ in \eqref{e3.20}, it
follows that
\[
(4-\lambda^2)\eta_0'(\lambda)+\lambda\eta_0(\lambda)=2,
\quad(\lambda\in\C\setminus\R).
\]
and subsequently by multiplication with $n^2$ that
\begin{equation}
\begin{split}
0&=\sum_{j=1}^kn^{-2j+2}\big[\eta_{j-1}'''(\lambda)
+(4-\lambda^2)\eta_j'(\lambda)+\lambda\eta_j(\lambda)\big]\\
&\phantom{\big[(4-\lambda^2)\eta_0'(\lambda)\big]}
+n^{-2k}\big[\eta_k'''(\lambda)+(4-\lambda^2)R_{k,n}'(\lambda)
+\lambda R_{k,n}(\lambda)\big]
+n^{-2k-2}R_{k,n}'''(\lambda).
\end{split}
\label{e3.21}
\end{equation}
Letting then $n\to\infty$ in \eqref{e3.21}, we find similarly
(assuming $k\ge1$) that
\[
\eta_{0}'''(\lambda)+(4-\lambda^2)\eta_1'(\lambda)+\lambda\eta_1(\lambda)=0,
\]
and subsequently that
\begin{equation*}\begin{split}
0&=\sum_{j=2}^kn^{-2j+4}\big[\eta_{j-1}'''(\lambda)
+(4-\lambda^2)\eta_j'(\lambda)+\lambda\eta_j(\lambda)\big]\\
&\phantom{\big[(4-\lambda^2)\eta_0'(\lambda)\big]}
+n^{-2k+2}\big[\eta_k'''(\lambda)+(4-\lambda^2)R_{k,n}'(\lambda)
+\lambda R_{k,n}(\lambda)\big]
+n^{-2k}R_{k,n}'''(\lambda).
\end{split}\end{equation*}
Continuing like this (induction), we obtain \eqref{e3.22} for any $j$
in $\{1,2,\ldots,k\}$. Since $k$ can be chosen arbitrarily in $\N$, we
obtain the desired conclusion.
\end{proof}

For any odd integer $k$ we shall in the following use the
conventions:
\begin{equation}
(\lambda^2-4)^{1/2}=\lambda\sqrt{1-\tfrac{4}{\lambda^2}}, \qand
(\lambda^2-4)^{k/2}=\big((\lambda^2-4)^{1/2}\big)^k
\label{e3.21a}
\end{equation}
for any $\lambda$ in the region
\[
\Omega:=\C\setminus[-2,2],
\]
and where $\sqrt{\cdot}$ denotes the usual main branch of the square root on
$\C\setminus(-\infty,0]$. We note in particular that
\begin{equation}
\big|(\lambda^2-4)^{1/2}\big|\to\infty, \quad\mbox{as $|\lambda|\to\infty$}.
\label{e3.29}
\end{equation}

\begin{lemma}\label{loesn af diff lign II}
For any $r$ in $\Z\setminus\{-3,-4\}$ the complete solution to the
differential equation:
\begin{equation}
(\lambda^2-4)\frac{\rd}{d\lambda}f(\lambda)-\lambda f(\lambda)
=\frac{\rd^3}{\rd\lambda^3}(\lambda^2-4)^{-r-1/2},
\qquad(\lambda\in\Omega)
\label{e3.23}
\end{equation}
is given by
\begin{equation}
f(\lambda)=\frac{(r+1)(2r+1)(2r+3)}{(r+3)(\lambda^2-4)^{r+5/2}}
+\frac{2(2r+1)(2r+3)(2r+5)}{(r+4)(\lambda^2-4)^{r+7/2}}
+C(\lambda^2-4)^{1/2},
\label{e3.24}
\end{equation}
for all $\lambda$ in $\Omega$, and where $C$ is an arbitrary
complex constant. 
\end{lemma}

\begin{proof} By standard methods the complete solution to \eqref{e3.23} is
given by
\begin{equation}
f(\lambda)=(\lambda^2-4)^{1/2}\int(\lambda^2-4)^{-3/2}
\tfrac{\rd^3}{\rd\lambda^3}(\lambda^2-4)^{-r-1/2}\6\lambda,
\qquad(\lambda\in\Omega),
\label{e3.31}
\end{equation}
where $\int(\lambda^2-4)^{-3/2}
\tfrac{\rd^3}{\rd\lambda^3}(\lambda^2-4)^{-r-1/2}\6\lambda$ denotes the
class of anti-derivatives (on $\Omega$) to the function
$(\lambda^2-4)^{-3/2}
\tfrac{\rd^3}{\rd\lambda^3}(\lambda^2-4)^{-r-1/2}$. Note here that by a
standard calculation,
\[
(\lambda^2-4)^{-3/2}\frac{\rd^3}{\rd\lambda^3}(\lambda^2-4)^{-r-1/2}
=\frac{-(2r+1)(2r+2)(2r+3)\lambda}{(\lambda^2-4)^{r+4}}
-\frac{4(2r+1)(2r+3)(2r+5)\lambda}{(\lambda^2-4)^{r+5}}.
\]
Assuming that $r\notin\{-3,-4\}$, we have (since $\Omega$ is
connected) for $k$ in $\{4,5\}$ that
\[
\int\lambda(\lambda^2-4)^{-r-k}\6\lambda=\frac{-1}{2(r+k-1)}(\lambda^2-4)^{-r-k+1}
+C, \qquad(C\in\C).
\]
We obtain thus that
\begin{equation*}\begin{split}
\int(\lambda^2-4)^{-3/2}
&\frac{\rd^3}{\rd\lambda^3}(\lambda^2-4)^{-r-1/2}\6\lambda\\[.2cm]
&=\frac{(2r+1)(2r+2)(2r+3)}{2(r+3)(\lambda^2-4)^{r+3}}+
\frac{4(2r+1)(2r+3)(2r+5)}{2(r+4)(\lambda^2-4)^{r+4}}+C\\[.2cm]
&=\frac{(r+1)(2r+1)(2r+3)}{(r+3)(\lambda^2-4)^{r+3}}+
\frac{2(2r+1)(2r+3)(2r+5)}{(r+4)(\lambda^2-4)^{r+4}}+C,
\end{split}\end{equation*}
where $C$ is an arbitrary constant.
Inserting this expression into \eqref{e3.31}, formula \eqref{e3.24}
follows readily. 
\end{proof}

\begin{proposition}\label{eksplicit udtryk for eta}
Let $G_n$ denote the Cauchy-transform of $h_n(x)\6x$,
  and consider for each $\lambda$ in $\C\setminus\R$ and $k$ in
  $\N_0$ the asymptotic expansion:
\[
G_n(\lambda)
=\eta_0(\lambda)+\frac{\eta_1(\lambda)}{n^2}+\frac{\eta_2(\lambda)}{n^4}+\cdots
+\frac{\eta_k(\lambda)}{n^{2k}}+O(n^{-2k-2})
\]
given by Corollary~\ref{asymptotisk udv.}. Then for $\lambda$ in
$\C\setminus\R$ we have that
\begin{align}
\eta_0(\lambda)&=\frac{\lambda}{2}-\frac{1}{2}(\lambda^2-4)^{1/2},
\label{e3.25}\\[.2cm]
\eta_1(\lambda)&=(\lambda^2-4)^{-5/2},\label{e3.26}
\end{align}
and generally for $j$ in $\N$, $\eta_j$ takes the form:
\[
\eta_j(\lambda)=\sum_{r=2j}^{3j-1}C_{j,r}(\lambda^2-4)^{-r-1/2}
\]
for constants $C_{j,r}$, $2j\le r\le 3j-1$.
Whenever $j\ge1$, these constants satisfy the recursion formula:
\begin{equation}
%C_{j+1,2j+2}&=&\frac{(4j+1)(4j+2)(4j+3)}{4j+6}C_{j,2j},\\[.2cm]
C_{j+1,r}=\frac{(2r-3)(2r-1)}{r+1}\big((r-1)C_{j,r-2}+(4r-10)C_{j,r-3}\big),
\quad(2j+2\le r\le 3j+2),
\label{e3.27a}
%\\[.2cm]
%C_{j+1,3j+2}&=&\frac{2(6j-1)(6j+1)(6j+3)}{3j+3}C_{j,3j-1}.
\end{equation}
where for $r$ in $\{2j+2,3j+2\}$ we adopt the conventions:
$C_{j,2j-1}=0=C_{j,3j}$.
\end{proposition}

Before proceeding to the proof of Proposition~\ref{eksplicit
    udtryk for eta}, we note that for any $j$ in $\N_0$ and $\lambda$
    in $\C\setminus\R$ we have by Lemma~\ref{kont i Frechet rum} that 
\[
|\eta_j(\lambda)|\le\|T^jg_\lambda\|_{\infty}\le
 C(j,0)\|g_\lambda\|_{m(j,0)},
\]
for suitable constants $C(j,0)$ in $(0,\infty)$ and $m(j,0)$ in $\N$
(not depending on $\lambda$). In particular it follows that
\begin{equation}
|\eta_j(\ri x)|\to0, \quad\mbox{as $x\to\infty$, \ $x\in\R$.}
\label{e3.28}
\end{equation}

\begin{proofof}[Proof of Proposition~\ref{eksplicit udtryk for eta}.]

The function $\eta_0$ is the Cauchy transform of the standard
semi-circle distribution, which is well-known to equal the right hand
side of \eqref{e3.25} (see e.g.\ \cite{vdn}).
Now, $\eta_0'''(\lambda)=
-\frac{1}{2}\frac{\rd^3}{\rd\lambda^3}(\lambda^2-4)^{1/2}$, 
so by \eqref{e3.22} and Lemma~\ref{loesn af diff lign II} (with
$r=-1$), it follows that
\[
\eta_1(\lambda)=
-\frac{1}{2}\big(-2(\lambda^2-4)^{1-7/2}\big)+C(\lambda^2-4)^{1/2}
=(\lambda^2-4)^{-5/2}+C(\lambda^2-4)^{1/2},
\]
for a suitable constant $C$ in $\C$. Comparing \eqref{e3.28} and
\eqref{e3.29}, it follows that we must have $C=0$, which establishes 
\eqref{e3.26}. 

Proceeding by induction, assume that for some
$j$ in $\N$ we have established that
\[
\eta_j(\lambda)=\sum_{r=2j}^{3j-1}C_{j,r}(\lambda^2-4)^{-r-1/2}
\]
for suitable constants $C(j,r)$, $r=2j,2j+1,\ldots,3j-1$. Then by
\eqref{e3.22}, Lemma~\ref{loesn af diff lign II} and linearity it
follows that modulo a term of the form $C(\lambda^2-4)^{1/2}$ we have
that 
\begin{equation*}\begin{split}
&\eta_{j+1}(\lambda)\\[.2cm]
&=\sum_{r=2j}^{3j-1}C_{j,r}
\tfrac{(r+1)(2r+1)(2r+3)}{r+3}(\lambda^2-4)^{-r-5/2}
+\sum_{r=2j}^{3j-1}C_{j,r}
\tfrac{2(2r+1)(2r+3)(2r+5)}{r+4}(\lambda^2-4)^{-r-7/2}\\[.2cm]
&=\sum_{s=2j+2}^{3j+1}C_{j,s-2}
\tfrac{(s-1)(2s-3)(2s-1)}{s+1}(\lambda^2-4)^{-s-1/2}
+\sum_{s=2j+3}^{3j+2}C_{j,s-3}
\tfrac{2(2s-5)(2s-3)(2s-1)}{s+1}(\lambda^2-4)^{-s-1/2}\\[.2cm]
&=C_{j,2j}\tfrac{(2j+1)(4j+1)(4j+3)}{2j+3}(\lambda^2-4)^{-2j-2-1/2}
+C_{j,3j-1}\tfrac{2(6j-1)(6j+1)(6j+3)}{3j+3}(\lambda^2-4)^{-3j-2-1/2}\\
&\qquad+\sum_{s=2j+3}^{3j+1}
\tfrac{(2s-3)(2s-1)}{s+1}\Big[(s-1)C_{j,s-2}+
(4s-10)C_{j,s-3}\Big](\lambda^2-4)^{-s-1/2}.
\end{split}\end{equation*}
As before \eqref{e3.28} and \eqref{e3.29} imply that the neglected
term $C(\lambda^2-4)$ must vanish anyway. The resulting
expression in the calculation above has the form
\[
\sum_{s=2(j+1)}^{3(j+1)-1}C_{j+1,s}(\lambda^2-4)^{-s-1/2},
\]
where the constants $C_{j+1,s}$ are immediately given by \eqref{e3.27a},
whenever $2j+3\le s\le 3j+1$. Recalling the convention that
$C_{j,2j-1}=0=C_{j,3j}$, it is easy to check that also
when $s=2j+2$ or $s=3j+2$, formula \eqref{e3.27a} produces, respectively,
the coefficients to $(\lambda^2-4)^{-2j-5/2}$ and
$(\lambda^2-4)^{-3j-5/2}$ appearing in the resulting expression above.
\end{proofof}

Using the recursion formula \eqref{e3.27a}, it follows easily that
\begin{align*}
\eta_2(\lambda)&= 21(\lambda^2-4)^{-9/2}+105(\lambda^2-4)^{-11/2}\\[.2cm]
\eta_3(\lambda)&=
1485(\lambda^2-4)^{-13/2}+18018(\lambda^2-4)^{-15/2}
+50050(\lambda^2-4)^{-17/2}.
\end{align*}

We close this section by identifying the functionals
$g\mapsto\frac{1}{2\pi}\int_{-2}^2[T^jg](t)\sqrt{4-t^2}\6t$ as
distributions (in the sense of L.~Schwarts). Before stating the
result, we recall that the Chebychev polynomials $T_0,T_1,T_2,\ldots$
of the first kind are the polynomials on $\R$ determined by the
relation: 
\begin{equation}
T_k(\cos\theta)=\cos(k\theta), \qquad(\theta\in[0,\pi], \ k\in\N_0).
\label{e3.28c}
\end{equation}

\begin{corollary}\label{distributioner}
For each $j$ in $\N_0$ consider the mapping $\alpha_j\colon
C_b^{\infty}(\R)\to\C$ given by
\[
\alpha_j(g)=\frac{1}{2\pi}\int_{-2}^2[T^jg](t)\sqrt{4-t^2}\6t,\qquad(g\in
C_b^{\infty}(\R)),
\]
where $T\colon C_b^{\infty}(\R)\to C_b^{\infty}(\R)$ is the linear
mapping introduced in Theorem~\ref{eksistens af T}. Consider in
addition for each $k$ in $\N_0$ the mapping $E_k\colon
C_b^{\infty}(\R)\to\C$ given by
\[
E_k(g)=\frac{1}{\pi}\int_{-2}^2g^{(k)}(x)
\frac{T_k(\tfrac{x}{2})}{\sqrt{4-x^2}}\6x,
\]
where $T_0,T_1,T_2,\ldots$ are the Chebychev polynomials given by
\eqref{e3.28c}. Then for any $j$ in $\N$,
\begin{equation}
\alpha_j=\sum_{k=2j}^{3j-1}C_{j,k}\frac{k!}{(2k)!}E_k,
\label{e3.28a}
\end{equation}
where $C_{j,2j},C_{j,2j+1},\ldots,C_{j,3j-1}$ are the constants
described in Proposition~\ref{eksplicit udtryk for eta}. 
\end{corollary}

From Corollary~\ref{distributioner}
it follows in particular that $\alpha_j$ (restricted to
$C_c^{\infty}(\R)$) is a distribution
supported on $[-2,2]$ (i.e., $\alpha_j(\phi)=0$ for any
function $\phi$ from $C_c^{\infty}(\R)$ such that
$\supp(\phi)\cap[-2,2]=\emptyset$). In addition it follows from
\eqref{e3.28a} that $\alpha_j$ is a distribution of order at most
$3j-1$ (cf.\ \cite[page 156]{ru}), and it is not hard to show that in fact
the order of $\alpha_j$ equals $3j-1$.

\begin{proofof}[Proof of Corollary~\ref{distributioner}.]
Let $j$ in $\N$ be given and let $\Lambda_j$ denote the right
hand side of \eqref{e3.28a}. Since both $\alpha_j$ and $\Lambda_j$ are
supported on $[-2,2]$, it suffices to show that their Stieltjes
transforms coincide, i.e., that
\begin{equation}
\alpha_j(g_\lambda)=\Lambda_j(g_\lambda),
\qquad(\lambda\in\C\setminus\R),
\label{e3.28b}
\end{equation}
where as before $g_\lambda(x)=\frac{1}{\lambda-x}$ for all $x$ in
$\R$. Since the mapping $\lambda\mapsto g_\lambda$ is analytical from
$\C\setminus\R$ into $C_b^{\infty}(\R)$ (cf.\ Lemma~\ref{analytisk}),
and since the linear functionals $\alpha_j,\Lambda_j\colon
C_b^{\infty}(\R)\to\C$ are continuous,
the functions $\lambda\mapsto\alpha_j(g_\lambda)$ and
$\lambda\mapsto\Lambda_j(g_\lambda)$ are analytical on
$\C\setminus\R$. It suffices thus to establish \eqref{e3.28b} for
$\lambda$ in $\C\setminus\R$ such that $|\lambda|>2$. So consider in the
following a fixed such $\lambda$.
We know from Proposition~\ref{eksplicit udtryk for eta} that
\[
\alpha_j(g_\lambda)=\eta_j(\lambda)=\sum_{k=2j}^{3j-1}C_{j,k}
(\lambda^2-4)^{-k-\frac{1}{2}},
\]
with $(\lambda^2-4)^{-k-\frac{1}{2}}$ defined as in \eqref{e3.21a}. It
suffices thus to show that
\[
E_k(g_\lambda)=\frac{(2k)!}{k!}(\lambda^2-4)^{-k-\frac{1}{2}}
\]
for all $k$ in $\N$. So let $k$ from $\N$ be given, and recall that
$g_\lambda(x)=\frac{1}{\lambda}\sum_{\ell=0}^{\infty}
(\frac{x}{\lambda})^{\ell}$ for all $x$ in $[-2,2]$. Since
$\int_{-2}^2\frac{|T_k(\frac{x}{2})|}{\sqrt{4-x^2}}\6x<\infty$, and
since the power series
\[
\sum_{\ell=r}^{\infty}\ell(\ell-1)\cdots(\ell-r+1)z^{\ell-r}
\]
converges uniformly on $\{z\in\C\mid |z|\le\frac{2}{|\lambda|}\}$ for
any $r$ in $\N_0$, it follows that we may change the order of
differentiation, summation and integration in the following
calculation: 
\begin{equation*}\begin{split}
E_k(g_\lambda)&=\frac{1}{\pi\lambda}\int_{-2}^2
\Big[\frac{\rd^k}{\rd x^k}\sum_{\ell=0}^{\infty}
\Big(\frac{x}{\lambda}\Big)^{\ell}\Big]
\frac{T_k(\tfrac{x}{2})}{\sqrt{4-x^2}}\6x\\[.2cm]
&=\frac{1}{\pi\lambda}\int_{-2}^2
\Big[\lambda^{-k}\sum_{\ell=k}^{\infty}\ell(\ell-1)\cdots(\ell-k+1)
\Big(\frac{x}{\lambda}\Big)^{\ell-k}\Big]
\frac{T_k(\tfrac{x}{2})}{\sqrt{4-x^2}}\6x\\[.2cm]
&=\sum_{\ell=k}^{\infty}\frac{\ell!}{(\ell-k)!}\Big[
\frac{1}{\pi}\int_{-2}^2x^{\ell-k}
\frac{T_k(\tfrac{x}{2})}{\sqrt{4-x^2}}\6x\Big]\lambda^{-\ell-1}.
\end{split}\end{equation*}
Using the substitution $x=2\cos\theta$, $\theta\in(0,\pi)$, as well as
\eqref{e3.28c} and Euler's formula for $\cos\theta$, it follows by a
standard calculation that
\[
\frac{1}{\pi}\int_{-2}^2x^{p}\frac{T_k(\tfrac{x}{2})}{\sqrt{4-x^2}}\6x=
\begin{cases}
\binom{p}{(p-k)/2}, &\mbox{if $p\in\{k+2m\mid m\in\N_0\}$},\\
0, &\mbox{otherwise}.
\end{cases}
\]
We thus find that
\begin{equation*}\begin{split}
E_k(g_\lambda)&=\sum_{m=0}^{\infty}\frac{(2k+2m)!}{(k+2m)!}
\binom{k+2m}{m}\lambda^{-2m-2k-1}
=\sum_{m=0}^{\infty}\frac{(2k+2m)!}{m!(k+m)!}\lambda^{-2m-2k-1}
\\[.2cm]
&=\frac{(2k)!}{k!}
\sum_{m=0}^{\infty}4^m\binom{k+m-\frac{1}{2}}{m}\lambda^{-2m-2k-1}
=\lambda^{-2k-1}\frac{(2k)!}{k!}
\sum_{m=0}^{\infty}\binom{k+m-\frac{1}{2}}{m}
\Big(\frac{4}{\lambda^2}\Big)^m,
\end{split}\end{equation*}
where the third equality results from a standard calculation on
binomial coefficients. Recall now that
\[
(1-z)^{-k-\frac{1}{2}}=\sum_{m=0}^{\infty}\binom{k+m-\frac{1}{2}}{m}z^m, 
\qquad(z\in\C, \ |z|<1),
\]
where the left hand side is formally defined as
$(\sqrt{1-z})^{-2k-1}$, with $\sqrt{\cdot}$ the usual holomorphic
branch of the square root on $\C\setminus(-\infty,0]$. We may thus
conclude that
\[
E_k(g_\lambda)=\lambda^{-2k-1}\frac{(2k)!}{k!}
\Big(1-\frac{4}{\lambda^2}\Big)^{-k-\frac{1}{2}}
=\frac{(2k)!}{k!}\Big(\lambda\sqrt{1-\tfrac{4}{\lambda^2}}\Big)^{-2k-1}
=\frac{(2k)!}{k!}\big(\lambda^2-4\big)^{-k-\frac{1}{2}},
\]
where the last equality follows from \eqref{e3.21a}. This completes
the proof.
\end{proofof}

\section{Asymptotic expansion for second order statistics}\label{2-dim
  udv} 

In this section we shall establish asymptotic expansions, similar to
Corollary~\ref{asymptotisk udv.}, for covariances in the form
$\Cov\{\Tr_n[f(X_n)],\Tr_n[g(Y_n)]\}$, where $f,g\in C_b^{\infty}(\R)$,
$X_n$ is a $\GUE(n,\frac{1}{n})$ random matrix and $\Tr_n$ denotes the
(un-normalized) trace on $M_n(\C)$.

For complex-valued random variables $Y,Z$ with second moments (and
defined on the same probability space), we use the notation:
\[
\V\{Y\}=\E\big\{(Y-\E\{Y\})^2\big\}, \qand
\Cov\{Y,Z\}=\E\big\{(Y-\E\{Y\})(Z-\E\{Z\})\big\}.
\]
Note in particular that $\V\{Y\}$ is generally not a positive number,
and that $\Cov\{Y,Z\}$ is truly linear in both $Y$ and $Z$.

\begin{lemma}\label{Pastur}
Let $\sigma$ be a positive number, and let $X_N$ be a
$\GUE(n,\sigma^2$) random matrix.
For any function $f$ from $C^{\infty}_b(\R)$ we then have that
\begin{equation}
\V\big\{\Tr_n[f(X_n)]\big\}=\frac{1}{4\sigma^2}\int_{\R^2}(f(x)-f(y))^2
\psi_{n}\big(\tfrac{x}{\sqrt{2\sigma^2}},
\tfrac{y}{\sqrt{2\sigma^2}}\big)^2\6x\6y,
\label{e4.1}
\end{equation}
where the kernel $\psi_n$ is given by
\begin{equation}
\psi_{n}(x,y)=\sum_{j=0}^{n-1}\varphi_j(x)\varphi_j(y)
=\sqrt{\frac{n}{2}}\frac{\varphi_n(x)\varphi_{n-1}(y) -
  \varphi_{n-1}(x)\varphi_{n}(y)}{x-y},
\label{e4.1a}
\end{equation}
and the $\varphi_j$'s are the Hermite functions introduced in \eqref{e3.1}.
\end{lemma}

\begin{proof} Formula \eqref{e4.1} appears in the proof of
\cite[Lemma~3]{ps} with $\psi$ given by the first equality in
\eqref{e4.1a}. The second equality in \eqref{e4.1a} is equivalent to
the Christoffel-Darboux formula for the Hermite polynomials (see
\cite[p.~193 formula (11)]{htf}).
\end{proof}

\begin{corollary}\label{Pastur for covarians}
Let $X_n$ be a $\GUE(n,\frac{1}{n})$ random matrix.

\begin{itemize}

\item[(i)] For any function $f$ from $C_b^{\infty}(\R)$ we have that
\[
\V\big\{\Tr_n[f(X)]\big\}=
\int_{\R^2}\Big(\frac{f(x)-f(y)}{x-y}\Big)^2\rho_n(x,y)\6x\6y,
\]
where the kernel $\rho_n$ is given by
\begin{equation}
\rho_n(x,y)=\tfrac{n}{4}\big[\varphi_n(\textstyle{\sqrt{\tfrac{n}{2}}x})
\varphi_{n-1}(\sqrt{\tfrac{n}{2}}y)-\varphi_{n-1}(\sqrt{\tfrac{n}{2}}x)
\varphi_n(\sqrt{\tfrac{n}{2}}y)\big]^2.
\label{e4.2}
\end{equation}

\item[(ii)] For any functions $f$ and $g$ from $C_b^{\infty}(\R)$ we have that
\[
\Cov\big\{\Tr_n[f(X_n)],\Tr_n[g(X_n)]\big\}=
\int_{\R^2}\Big(\frac{f(x)-f(y)}{x-y}\Big)
\Big(\frac{g(x)-g(y)}{x-y}\Big)\rho_n(x,y)\6x\6y.
\]
\end{itemize}
\end{corollary}

\begin{proof} (i) \ This follows from Lemma~\ref{Pastur} by a straightforward
calculation, setting $\sigma^2=\frac{1}{n}$ in \eqref{e4.1}.

(ii) \ Using (i) on the functions $f+g$ and $f-g$ we find that
\begin{equation*}\begin{split}
\Cov\big\{&\Tr_n[f(X_n)],\Tr_n[g(X_n)]\big\}\\[.2cm]
&=\frac{1}{4}
\Big(\V\big\{\Tr_n[f(X_n)+g(X_n)]\big\}
-\V\big\{\Tr_n[f(X_n)-g(X_n)]\big\}\Big)
\\[.2cm]
&=\frac{1}{4}\int_{\R^2}\frac{\big((f+g)(x)-(f+g)(y)\big)^2
-\big((f-g)(x)-(f-g)(y)\big)^2}{(x-y)^2}\rho_n(x,y)\6x\6y\\[.2cm]
&=\frac{1}{4}\int_{\R^2}\frac{4f(x)g(x)+4f(y)g(y)-4f(x)g(y)-4f(y)g(x)}
{(x-y)^2}\rho_n(x,y)\6x\6y\\[.2cm]
&=\int_{\R^2}\Big(\frac{f(x)-f(y)}{x-y}\Big)
\Big(\frac{g(x)-g(y)}{x-y}\Big)\rho_n(x,y)\6x\6y,
\end{split}\end{equation*}
as desired.
\end{proof}

In order to establish the desired asymptotic expansion of
$\Cov\{\Tr_n[f(X_n)],\Tr_n[g(X_n)]\}$, we are lead by
Corollary~\ref{Pastur for covarians}(ii)
to study the asymptotic behavior, as $n\to\infty$, of the
probability measures $\rho_n(x,y)\6x\6y$. As a first step, it is
instructive to note that $\rho_n(x,y)\6x\6y$ converges weakly, as
$n\to\infty$, to the probability measure $\rho(x,y)\6x\6y$, where
\begin{equation}
\rho(x,y)=\frac{1}{4\pi^2}\frac{4-xy}{\sqrt{4-x^2}\sqrt{4-y^2}}
1_{(-2,2)}(x)1_{(-2,2)}(y).
\label{e4.3a}
\end{equation}
We shall give a short proof of this fact in 
Proposition~\ref{svag konvergens af rho-n} below. 
It implies in particular that if $(X_n)$
is a sequence of random matrices, such that $X_n\sim
\GUE(n,\frac{1}{n})$ for all $n$, then
\[
\lim_{n\to\infty}\Cov\big\{\Tr_n[f(X_n)],\Tr_n[g(X_n)]\big\}
=\int_{\R^2}\Big(\frac{f(x)-f(y)}{x-y}\Big)
\Big(\frac{g(x)-g(y)}{x-y}\Big)\rho(x,y)\6x\6y,
\]
for all $f,g\in C_b^{\infty}(\R)$.

The key point in the approach given below is to express
the density $\rho_n$ in terms of the spectral density $h_n$ of
$\GUE(n,\frac{1}{n})$ (see Proposition~\ref{formel for rho-n} below).

\begin{lemma}\label{formel for zeta-n}
Consider the functions $\zeta_n\colon\R^2\to\R$ and
$\beta_n\colon\R\to\R$ given by
\[
\zeta_n(x,y)=\frac{1}{2}\big[\varphi_n(x)\varphi_{n-1}(y)-
\varphi_{n-1}(x)\varphi_n(y)\big]^2,\qquad((x,y)\in\R^2),
\]
and
\[
\beta_n(x)=\sum_{j=0}^{n-1}\varphi_j(x)^2, \qquad(x\in\R),
\]
with $\phi_0,\phi_1,\phi_2,\ldots$ the Hermite functions given in
\eqref{e3.1}. We then have
\[
\zeta_n(x,y)=f_n(x)f_n(y)-g_n(x)g_n(y)-k_n(x)k_n(y),
\qquad((x,y)\in\R^2),
\]
where
\begin{align}
f_n(x)&=\tfrac{1}{2}\big(\varphi_n(x)^2+\varphi_{n-1}(x)^2\big)
 = \tfrac{1}{2n}\big(\beta_n(x)-x\beta_n'(x)\big),\label{e4.3}\\[.2cm]
g_n(x)&=\tfrac{1}{2}\big(\varphi_n(x)^2-\varphi_{n-1}(x)^2\big)
 = \tfrac{1}{4n}\beta_n''(x),\label{e4.4}\\[.2cm]
k_n(x)&=\varphi_{n-1}(x)\varphi_n(x) = \tfrac{-1}{\sqrt{2n}}\beta_n'(x)
\label{e4.5}
\end{align}
for all $x$ in $\R$.
\end{lemma}

\begin{proof} Note first that with $f_n,g_n$ and $k_n$ defined by the leftmost
equalities in \eqref{e4.3}-\eqref{e4.5} we have that
\[
f_n(x)+g_n(x)=\varphi_n(x)^2 \qand f_n(x)-g_n(x)=\varphi_{n-1}(x)^2,
\]
for all $x$ in $\R$. Therefore,
\begin{equation*}\begin{split}
\zeta_n(x,y)&=\frac{1}{2}\big[\varphi_n(x)\varphi_{n-1}(y)-
\varphi_{n-1}(x)\varphi_n(y)\big]^2\\[.2cm]
&=\frac{1}{2}\big[(f_n(x)+g_n(x))(f_n(y)-g_n(y))
+(f_n(x)-g_n(x))(f_n(y)+g_n(y))-2k_n(x)k_n(y)\big]\\[.2cm]
&=f_n(x)f_n(y)-g_n(x)g_n(y)-k_n(x)k_n(y),
\end{split}\end{equation*}
for any $(x,y)$ in $\R^2$. It remains thus to establish the three
rightmost equalities in \eqref{e4.3}-\eqref{e4.5}. For this we use the
well-known formulas (cf.\ e.g.\ \cite[formulas (2.3)-(2.6)]{ht1}):
\begin{align}
\varphi_n'(x)&={\textstyle \sqrt{\frac{n}{2}}\varphi_{n-1}(x)-
\sqrt{\frac{n+1}{2}}\varphi_{n+1}(x)}, 
\label{e4.6} \\
x\varphi_n(x)&=\textstyle{\sqrt{\frac{n+1}{2}}\varphi_{n+1}(x)+
\sqrt{\frac{n}{2}}\varphi_{n-1}(x)}, 
\label{e4.8}\\
\frac{\rd}{\rd x}\Big(\sum_{k=0}^{n-1}\varphi_k(x)^2\Big) &=
-\sqrt{2n}\varphi_n(x)\varphi_{n-1}(x), 
\label{e4.7}
\end{align}
which hold for all $n$ in $\N_0$, when we adopt the convention:
$\varphi_{-1}\equiv 0$.

The second equality in \eqref{e4.5} is an immediate consequence of
\eqref{e4.7}. Combining \eqref{e4.6} with \eqref{e4.8}, we note next
that
\[
\varphi_n'(x)=-x\varphi_n(x)+\sqrt{2n}\varphi_{n-1}(x) \qand
\varphi'_{n-1}(x)=x\varphi_{n-1}(x)-\sqrt{2n}\varphi_n(x),
\]
and therefore by \eqref{e4.7}
\begin{equation*}\begin{split}
\beta_n''(x)&=-\sqrt{2n}\big(\varphi_{n-1}'(x)\varphi_n(x)+
\varphi_{n-1}(x)\varphi_n'(x)\big)\\[.2cm]
&=-\sqrt{2n}\big(x\varphi_{n-1}(x)\varphi_n(x)-\sqrt{2n}\varphi_n(x)^2
-x\varphi_{n-1}(x)\varphi_n(x)+\sqrt{2n}\varphi_{n-1}(x)^2\big)\\[.2cm]
&=2n\big(\varphi_n(x)^2-\varphi_{n-1}(x)^2\big),
\end{split}\end{equation*}
from which the second equality in \eqref{e4.4} follows readily.
Using once more \eqref{e4.6} and \eqref{e4.8}, we note finally that
\begin{equation*}\begin{split}
x\beta_n'(x)&=2\sum_{j=0}^{n-1}x\varphi_j(x)\varphi_j'(x)\\[.2cm]
&=2\sum_{j=0}^{n-1}\Big(\sqrt{\tfrac{j+1}{2}}\varphi_{j+1}(x)
+\sqrt{\tfrac{j}{2}}\varphi_{j-1}(x)\Big)
\Big(\sqrt{\tfrac{j}{2}}\varphi_{j-1}(x)
-\sqrt{\tfrac{j+1}{2}}\varphi_{j+1}(x)\Big)\\[.2cm]
&=\sum_{j=0}^{n-1}\big(\varphi_{j-1}(x)^2+(j-1)\varphi_{j-1}(x)^2
-(j+1)\varphi_{j+1}(x)^2\big)\\[.2cm]
&=\Big(\sum_{j=0}^{n-2}\varphi_j(x)^2\Big)
-(n-1)\varphi_{n-1}(x)^2-n\varphi_n(x)^2,
\end{split}\end{equation*}
and therefore
\[
\beta_n(x)-x\beta_n'(x)
=\varphi_{n-1}(x)^2+(n-1)\varphi_{n-1}(x)^2+n\varphi_n(x)^2
=n\big(\varphi_{n-1}(x)^2+\varphi_n(x)^2\big),
\]
which establishes the second equality in \eqref{e4.3}.
\end{proof}

\begin{proposition}\label{formel for rho-n}
Let $\rho_n$ be the kernel given by \eqref{e4.2} and let $h_n$ be the
spectral density of a $\GUE(n,\frac{1}{n})$ random matrix (cf.\
\eqref{e3.3}). We then have
\begin{equation}
\rho_n(x,y)=\frac{1}{4}\big[\tilde{h}_n(x)\tilde{h}_n(y)-4h_n'(x)h_n'(y)
-\tfrac{1}{n^2}h_n''(x)h_n''(y)\big], \qquad((x,y)\in\R^2),
\label{e4.11}
\end{equation}
where
\[
\tilde{h}_n(x)=h_n(x)-xh_n'(x), \qquad(x\in\R).
\]
\end{proposition}

\begin{proof} With $\zeta_n,f_n,g_n,k_n$ and $\beta_n$ as in Lemma~\ref{formel for
  zeta-n} we have that
\begin{equation}
\begin{split}
\rho_n(x,y)&=\tfrac{n}{2}
\zeta_n\big(\textstyle{\sqrt{\frac{n}{2}}x,\sqrt{\frac{n}{2}}y}\big)
\\[.2cm]
&=\tfrac{n}{2}\big(f_n\big(\textstyle{\sqrt{\tfrac{n}{2}}x}\big)
f_n\big(\textstyle{\sqrt{\tfrac{n}{2}}y}\big)
-g_n\big(\textstyle{\sqrt{\tfrac{n}{2}}x}\big)
g_n\big(\textstyle{\sqrt{\tfrac{n}{2}}y}\big)
-k_n\big(\textstyle{\sqrt{\tfrac{n}{2}}x}\big)
k_n\big(\textstyle{\sqrt{\tfrac{n}{2}}y}\big)\big),
\end{split}
\label{e4.9}
\end{equation}
and (cf.\ formula \eqref{e3.3})
\begin{equation}
h_n(x)=\tfrac{1}{\sqrt{2n}}\beta_n\big(\textstyle{\sqrt{\frac{n}{2}}x}\big).
\label{e4.10}
\end{equation}
Combining \eqref{e4.10} with the rightmost equalities in
\eqref{e4.3}-\eqref{e4.5}, we find that
\[
f_n\big(\textstyle{\sqrt{\tfrac{n}{2}}x}\big)
=\tfrac{1}{\sqrt{2n}}\tilde{h}_n(x),\quad
g_n\big(\textstyle{\sqrt{\tfrac{n}{2}}x}\big)
=\frac{1}{\sqrt{2}n^{3/2}}h_n''(x), \qand
k_n\big(\textstyle{\sqrt{\tfrac{n}{2}}x}\big)
=-\textstyle{\sqrt{\tfrac{2}{n}}}h_n'(x),
\]
and inserting these expressions into \eqref{e4.9}, formula
\eqref{e4.11} follows readily.
\end{proof}

By $C_b^{\infty}(\R^2)$ we denote the vector space of infinitely often
differentiable functions $f\colon\R^2\to\C$ satisfying that
\[
\|D_1^kD_2^lf\|_{\infty}:=\sup_{(x,y)\in\R^2}\big|D_1^kD_2^lf(x,y)\big|<\infty,
\]
for any $k,l$ in $\N_0$. Here $D_1$ and $D_2$ denote, respectively, the
partial derivatives of $f$ with respect to the first and the second
variable.

\begin{lemma}\label{diff af snitfunktion} Assume that
  $f\in C_b^{\infty}(\R^2)$ and consider the mapping
  $\phi_f\colon\R\to C_b^{\infty}(\R)$ given by
\[
\phi_f(x)=f(x,\cdot), \qquad(x\in\R).
\]
Then $\phi_f$ is infinitely often differentiable from $\R$ into
$C_b^{\infty}(\R)$, and for any $k$ in $\N$
\begin{equation}
\frac{\rd^k}{\rd x^k}\phi_f(x)=\big[D_1^kf\big](x,\cdot),
\qquad(x\in\R).
\label{e4.10a}
\end{equation}
\end{lemma}

\begin{proof} By splitting $f$ in its real- and imaginary parts, we may
assume that $f$ is real-valued. For any $k$ in $\N$ the function
$D_1^kf$ is again an element of $C^{\infty}(\R^2)$. Therefore, by
induction, it suffices to prove that $\phi_f$ is differentiable with
derivative given by \eqref{e4.10a} (in the case $k=1$). For this we need to
establish that
\[
\Big\|\frac{\phi_f(x+h)-\phi_f(x)}{h}-[D_1f](x,\cdot)\Big\|_{(m)}
\longrightarrow0,\quad\mbox{as $h\to0$}
\]
for any $m$ in $\N$ and any $x$ in $\R$. This amounts to showing that
for fixed $x$ in $\R$ and $l$ in $\N$ we have that
\[
\sup_{y\in\R}\Big|\frac{D_2^lf(x+h,y)-D_2^lf(x,y)}{h}-D_2^lD_1f(x,y)\Big|
\longrightarrow0,\quad\mbox{as $h\to0$}.
\]
For fixed $y$ in $\R$ second order Taylor expansion for the
function $[D_2^lf](\cdot,y)$ yields that
\[
D_2^lf(x+h,y)-D_2^lf(x,y)=D_1D_2^lf(x,y)h+
\tfrac{1}{2}D_1^2D_2^lf(\xi,y)h^2,
\]
for some real number $\xi=\xi(x,y,h)$ between $x+h$ and
$x$. Consequently,
\[
\sup_{y\in\R}\Big|\frac{D_2^lf(x+h,y)-D_2^lf(x,y)}{h}-D_2^lD_1f(x,y)\Big|
\le\frac{h}{2}\big\|D_1^2D_2^lf\big\|_{\infty}\longrightarrow0,
\quad\mbox{as $h\to0$},
\]
as desired.
\end{proof}

\begin{corollary}\label{diff T af snitfunktion}
Let $T$ be the linear mapping introduced in Remark~\ref{def af S og
  T}, and let $f$ be a function from $C_b^{\infty}(\R^2)$. We then have

\begin{itemize}

\item[(i)]
For any $j$ in $\N_0$ the mapping
\[
\psi_f\colon x\mapsto T^jf(x,\cdot)\colon\R\to C_b^{\infty}(\R)
\]
is infinitely often differentiable with derivatives given by
\begin{equation}
\frac{\rd^k}{\rd x^k}\psi_f(x)=T^j\big([D_1^kf](x,\cdot)\big).
\label{e4.10b}
\end{equation}

\item[(ii)] For any $j$ in $\N_0$ the mapping $\upsilon_j\colon\R\to\C$
  given by
\[
\upsilon_j(x)=\frac{1}{2\pi}\int_{-2}^2\big[T^jf(x,\cdot)\big](t)\sqrt{4-t^2}\6t,
\qquad(x\in\R),
\]
is a $C_b^{\infty}(\R)$-function. Moreover, for any $k$ in $\N$
\begin{equation}
\frac{\rd^k}{\rd x^k}\upsilon_j(x)
=\frac{1}{2\pi}\int_{-2}^2\big[T^j\big([D_1^k]f(x,\cdot)\big)\big](t)
\sqrt{4-t^2}\6t.
\label{e4.10c}
\end{equation}
\end{itemize}
\end{corollary}

\begin{proof} (i) \ As in the proof of Lemma~\ref{diff af snitfunktion} it
suffices to prove that $\psi_f^j$ is differentiable with derivative
given by \eqref{e4.10b} (in the case $k=1$). But this follows
immediately from Lemma~\ref{diff af snitfunktion}, since
$\psi_f=T^j\circ\phi_j$, where $T\colon C_b^{\infty}(\R)\to C_b^{\infty}(\R)$
  is a linear, continuous mapping (cf.\ Proposition~\ref{kontinuitet
    af S}).

(ii) \ It suffices to prove that $\upsilon_j$ is bounded and
differentiable with derivative given by \eqref{e4.10c} (in the case
$k=1$). To prove that $\upsilon_j$ is differentiable with the prescribed
derivative, it suffices, in view of (i), to establish
that the mapping
\[
g\mapsto\frac{1}{2\pi}\int_{-2}^2g(t)\sqrt{4-t^2}\6t\colon
C_b^{\infty}(\R)\to\R
\]
is linear and continuous. It is clearly linear, and since
\[
\Big|\frac{1}{2\pi}\int_{-2}^2g(t)\sqrt{4-t^2}\6t\Big|\le\|g\|_{\infty},
\qquad(g\in C_b^{\infty}(\R)),
\]
it is also continuous. To see finally that $\upsilon_j$ is a bounded
mapping, we note that since $T^j\colon C_b^{\infty}(\R)\to
C_b^{\infty}(\R)$ is continuous, there are (cf.\ Lemma~\ref{kont i
  Frechet rum}) constants $C$ from $(0,\infty)$ and $m$ in $\N$, such
that 
\[
\big\|T^jf(x,\cdot)\big\|_{\infty}\le
C\max_{l=1,\ldots,m}\big\|D_2^lf(x,\cdot)\big\|_{\infty}
\le C\max_{l=1,\ldots,m}\big\|D_2^lf\big\|_{\infty}
\]
for any $x$ in $\R$. Therefore,
\[
\sup_{x\in\R}\big|\upsilon_j(x)\big|\le
\sup_{x\in\R}\big\|T^jf(x,\cdot)\big\|_{\infty}
\le C\max_{l=1,\ldots,m}\big\|D_2^lf\big\|_{\infty}<\infty,
\]
since $f\in C_b^{\infty}(\R^2)$.
\end{proof}

\begin{proposition}\label{2-dim exp I}
For any function $f$ in $C_b^{\infty}(\R^2)$ there exists a sequence
$(\beta_j(f))_{j\in\N_0}$ of complex numbers such that 
\[
\int_{\R^2}f(x,y)h_n(x)h_n(y)\6x\6y
=\sum_{j=0}^k\frac{\beta_j(f)}{n^{2j}}+O(n^{-2k-2})
\]
for any $k$ in $\N_0$.
\end{proposition}

\begin{proof} Let $k$ in $\N_0$ be given.
For fixed $x$ in $\R$ the function $f(x,\cdot)$ belongs
to $C_b^{\infty}(\R^2)$ and hence Corollary~\ref{asymptotisk udv.}
asserts that
\begin{equation}
\int_{\R}f(x,y)h_n(y)\6y
=\sum_{j=1}^k\frac{\upsilon_j(x)}{n^{2j}}
+\frac{1}{n^{2k+2}}\int_{\R}\big[T^{k+1}f(x,\cdot)\big](t)h_n(t)\6t,
\label{e4.10d}
\end{equation}
where the functions $\upsilon_j\colon\R\to\R$ are given by
\[
\upsilon_j(x)=\frac{1}{2\pi}\int_{-2}^2\big[T^jf(x,\cdot)\big](t)
\sqrt{4-t^2}\6t, \qquad(x\in\R,\ j=1,\ldots,k).
\]
As noted in the proof of Corollary~\ref{diff T af snitfunktion},
there exist constants $C$ from $(0,\infty)$ and $m$ in $\N$, such
that
\[
\big\|T^{k+1}f(x,\cdot)\big\|_{\infty}
\le C\max_{l=1,\ldots,m}\big\|D_2^lf\big\|_{\infty}, \quad(x\in\R).
\]
Hence, since $h_n$ is a probability density, 
\[
C_k^f:=\sup_{x\in\R}\Big|\int_{\R}\big[T^{k+1}f(x,\cdot)\big](t)h_n(t)\6t\Big|
\le C\max_{l=1,\ldots,m}\big\|D_2^lf\big\|_{\infty}<\infty.
\]
Using now Fubini's Theorem and \eqref{e4.10d} we find that
\begin{equation}
\begin{split}
\int_{\R^2}f(x,y)h_n(x)h_n(y)\6x\6y
&=\int_{\R}\Big(\int_{\R}f(x,y)h_n(y)\6y\Big)h_n(x)\6x\\[.2cm]
&=\sum_{j=0}^kn^{-2j}\int_{\R}\upsilon_j(x)h_n(x)\6x+O(n^{-2k-2}),
\end{split}
\label{e4.10e}
\end{equation}
where the $O(n^{-2k-2})$-term is bounded by $C_k^fn^{-2k-2}$.
According to Corollary~\ref{diff T af snitfunktion}(ii), $\upsilon_j\in
C_b^{\infty}(\R)$ for each $j$ in $\{0,1,\ldots,k\}$, and hence
another application of Corollary~\ref{asymptotisk udv.} yields that
\[
\int_{\R}\upsilon_j(x)h_n(x)\6x
=\sum_{l=0}^{k-j}\frac{\xi_l^j(f)}{n^{2l}}+O(n^{-2k+2j-2}),
\]
for suitable complex numbers
$\xi_0^j(f),\ldots,\xi_{k-j}^j(f)$. Inserting these expressions into
\eqref{e4.10e} we find that
\begin{equation*}\begin{split}
\int_{\R^2}f(x,y)h_n(x)h_n(y)\6x\6y
&=\sum_{j=0}^k\Big(\sum_{l=0}^{k-j}\frac{\xi_l^j(f)}{n^{2(l+j)}}
+O(n^{-2k-2})\Big)+O(n^{-2k-2})\\[.2cm]
&=\sum_{r=0}^kn^{-2r}\Big(\sum_{j=0}^r\xi_{r-j}^j(f)\Big)+O(n^{-2k-2}).
\end{split}\end{equation*}
Thus, setting $\beta_r(f)=\sum_{j=0}^r\xi_{r-j}^j(f)$,
$r=0,1,\ldots,k$, we have obtained the desired expansion.
\end{proof}

For the proof of Theorem~\ref{2-dim exp III} below we need to extend
the asymptotic expansion in Proposition~\ref{2-dim exp I} to a larger
class of functions than $C_b^{\infty}(\R^2)$.

\begin{proposition}\label{2-dim exp II}
Assume that $f\colon\R^2\to\C$ is infinitely often differentiable, and
polynomially bounded in the sense that 
\[
|f(x,y)|\le C(1+x^2+y^2)^m, \qquad((x,y)\in\R^2)
\]
for suitable constants $C$ from $(0,\infty)$ and $m$ in $\N_0$.
Then there exists a sequence $(\beta_j(f))_{j\in\N_0}$ of complex
numbers, such that
\[
\int_{\R^2}f(x,y)h_n(x)h_n(y)\6x\6y
=\sum_{j=0}^k\frac{\beta_j(f)}{n^{2j}}+O(n^{-2k-2})
\]
for any $k$ in $\N_0$.
\end{proposition}

\begin{proof} We start by choosing a function $\phi$ from
$C^{\infty}_c(\R^2)$, satisfying that

\begin{itemize}

\item $\phi(x,y)\in[0,1]$ for all $(x,y)$ in $\R^2$.

\item $\supp(f)\subseteq[-4,4]\times[-4,4]$.

\item $\phi\equiv1$ on $[-3,3]\times[-3,3]$.

\end{itemize}

We then write $f=f\phi+f(1-\phi)$. Since $f\phi\in
C_c^{\infty}(\R^2)\subseteq C_b^{\infty}(\R^2)$, it follows from
Proposition~\ref{2-dim exp I} that there exists a sequence
$(\beta_j(f))_{j\in\N_0}$ of complex numbers, such that
\[
\int_{\R^2}f(x,y)\phi(x,y)h_n(x)h_n(y)\6x\6y
=\sum_{j=0}^k\frac{\beta_j(f)}{n^{2j}}+O(n^{-2k-2})
\]
for any $k$ in $\N_0$. Therefore, it suffices to establish that
\[
\int_{\R^2}f(x,y)(1-\phi(x,y))h_n(x)h_n(y)\6x\6y=O(n^{-2k-2})
\]
for any $k$ in $\N_0$. Note here that $(1-\phi)\equiv0$ on
$[-3,3]\times[-3,3]$, and that for some positive constant $C'$ we have that
\[
\big|f(x,y)(1-\phi(x,y))\big|\le C(1+x^2+y^2)^m\le
C'(x^{2m}+y^{2m})\le C'x^{2m}y^{2m},
\]
for all $(x,y)$ outside $[-3,3]\times[-3,3]$. Therefore,
\begin{equation*}\begin{split}
\int_{\R^2}f(x,y)(1-\phi(x,y))h_n(x)h_n(y)\6x\6y
&\le
C'\int_{\R^2\setminus[-3,3]\times[-3,3]}x^{2m}y^{2m}h_n(x)h_n(y)\6x\6y\\[.2cm]
&\le 4C'\int_{\R}\int_3^{\infty}x^{2m}y^{2m}h_n(x)h_n(y)\6y\6x\\[.2cm]
&=4C'\Big(\int_{\R}x^{2m}h_n(x)\6x\Big)\Big(\int_3^{\infty}y^{2m}h_n(y)\6y\Big),
\end{split}\end{equation*}
where the second estimate uses symmetry of the function $(x,y)\mapsto
x^{2m}y^{2m}h_n(x)h_n(y)$. By Wigner's semi-circle law (for moments)
\[
\lim_{n\to\infty}\int_{\R}x^{2m}h_n(x)\6x
=\frac{1}{2\pi}\int_{-2}^2x^{2m}\sqrt{4-x^2}\6x,
\]
and therefore it now suffices to show that
\begin{equation}
\int_3^{\infty}y^{2m}h_n(y)\6y=O(n^{-2k-2}) \qquad\mbox{for any $k$
  in $\N_0$}.
\label{e4.10f}
\end{equation}
Recall here that $h_n$ is the spectral density of a $\GUE(n,\frac{1}{n})$
random matrix $X_n$, so that
\begin{equation*}\begin{split}
\int_3^{\infty}y^{2m}h_n(y)\6y=\E\big\{\tr_n
\big[(X_n)^{2m}1_{(3,\infty)}(X_n)\big]\big\} 
&=\frac{1}{n}\E\Big\{\sum_{j=1}^n(\lambda_j^{(n)})^{2m}
1_{(3,\infty)}(\lambda_j^{(n)})\Big\},
\end{split}\end{equation*}
where $\lambda_1^{(n)}\le\lambda_2^{(n)}\le\cdots\le\lambda_n^{(n)}$
are the ordered (random) eigenvalues of $X_n$. Since the function
$y\mapsto y^{2m}1_{(3,\infty)}(y)$ is non-decreasing on $\R$, it follows
that
\[
\frac{1}{n}\sum_{j=1}^n(\lambda_j^{(n)})^{2m}1_{(3,\infty)}(\lambda_j^{(n)})
\le(\lambda_n^{(n)})^{2m}1_{(3,\infty)}(\lambda_n^{(n)})
\le\|X_n\|^{2m}1_{(3,\infty)}\big(\|X_n\|\big).
\]
Using \cite[Proposition~6.4]{hst} it thus follows that
\[
\int_3^{\infty}y^{2m}h_n(y)\6y\le
\E\big\{\|X_n\|^{2m}1_{(3,\infty)}\big(\|X_n\|\big)\big\}
\le \gamma(2m)n\e^{-n/2},
\]
for a suitable positive constant $\gamma(2m)$ (not depending on
$n$). This clearly implies \eqref{e4.10f}, and the proof is completed.
\end{proof}

\begin{theorem}\label{2-dim exp III}
Let $\rho_n$ be the kernel given by \eqref{e4.2}.
Then for any function $f$ in $C_b^{\infty}(\R^2)$ there exists a
sequence $(\beta_j(f))_{j\in\N_0}$ of complex numbers such that 
\[
\int_{\R^2}f(x,y)\rho_n(x,y)\6x\6y
=\sum_{j=0}^k\frac{\beta_j(f)}{n^{2j}}+O(n^{-2k-2})
\]
for any $k$ in $\N_0$.
\end{theorem}

\begin{proof} Using Proposition~\ref{formel for rho-n} we have that
\begin{equation}
\begin{split}
\int_{\R^2}f(x,y)\rho_n(x,y)\6x\6y
=\frac{1}{4}\int_{\R^2}f(x,y)\tilde{h}_n(x)\tilde{h}_n(y)\6x\6y
-&\int_{\R^2}f(x,y)h_n'(x)h_n'(y)\6x\6y\\
-\frac{1}{4n^{2}}&\int_{\R^2}f(x,y)h_n''(x)h_n''(y)\6x\6y,
\end{split}
\label{a4.1}
\end{equation}
and it suffices then to establish asymptotic expansions of the type
set out in the theorem for each of the integrals appearing on the
right hand side. 

By Fubini's Theorem and integration by parts, it follows that
\begin{equation}
\int_{\R^2}f(x,y)h_n'(x)h_n'(y)\6x\6y
=\int_{\R^2}\frac{\partial^2}{\partial x\partial
  y}f(x,y)h_n(x)h_n(y)\6x\6y,
\label{a4.2}
\end{equation}
and since $\frac{\partial^2}{\partial x\partial y}f(x,y)\in
C_b^{\infty}(\R^2)$, Proposition~\ref{2-dim exp I} yields an
asymptotic expansion of the desired kind for this integral. Similarly
\begin{equation}
\int_{\R^2}f(x,y)h_n''(x)h_n''(y)\6x\6y
=\int_{\R^2}\frac{\partial^4}{\partial x^2\partial y^2}f(x,y)
h_n(x)h_n(y)\6x\6y,
\label{a4.3}
\end{equation}
where $\frac{\partial^4}{\partial x^2\partial y^2}f(x,y)\in
C_b^{\infty}(\R^2)$, and another application of Proposition~\ref{2-dim
  exp I} yields the desired asymptotic expansion. Finally, using again
Fubini's Theorem and integration by parts,
\begin{equation}
\begin{split}
&\int_{\R^2}f(x,y)\tilde{h}_n(x)\tilde{h}_n(y)\6x\6y\\[.2cm]
&=\int_{\R^2}f(x,y)(h_n(x)-xh_n'(x))(h_n(y)-yh_n'(y))\6x\6y
\\[.2cm]
&=\int_{\R^2}f(x,y)\big[h_n(x)h_n(y)-xh_n'(x)h_n(y)-yh_n'(y)h_n(x)
+xyh_n'(x)h_n'(y)\big]\6x\6y\\[.2cm]
&=\int_{\R^2}\big[f(x,y)+\tfrac{\partial}{\partial x}(xf(x,y))
+\tfrac{\partial}{\partial y}(yf(x,y))
+\tfrac{\partial^2}{\partial x\partial y}(xyf(x,y))\big]
h_n(x)h_n(y)\6x\6y\\[.2cm]
&=\int_{\R^2}\big[4f(x,y)+2x\tfrac{\partial}{\partial x}f(x,y)
+2y\tfrac{\partial}{\partial y}f(x,y)
+xy\tfrac{\partial^2}{\partial x\partial y}f(x,y)\big]
h_n(x)h_n(y)\6x\6y.
\end{split}
\label{a4.4}
\end{equation}
In the latter integral, the function inside the brackets
is clearly a polynomially bounded $C^{\infty}$-function on $\R^2$, and
hence Proposition~\ref{2-dim exp II} provides an asymptotic expansion
of the desired kind. This completes the proof.
\end{proof}

\begin{corollary}\label{2-dim exp IV}
For any functions $f,g$ in $C^{\infty}_b(\R)$, there exists a sequence
$(\beta_j(f,g))_{j\in\N}$ of complex numbers, such that for any $k$ in $\N_0$
\begin{equation}
\begin{split}
\Cov\big\{\Tr_n[f(X_n)],\Tr_n[g(X_n)]\big\}
&=\int_{\R^2}\Big(\frac{f(x)-f(y)}{x-y}\Big)
\Big(\frac{g(x)-g(y)}{x-y}\Big)\rho_n(x,y)\6x\6y\\[.2cm]
&=\sum_{j=0}^k\frac{\beta_j(f,g)}{n^{2j}}+O(n^{-2k-2}).
\label{e4.10g}
\end{split}
\end{equation}
\end{corollary}

\begin{proof} The first equality in \eqref{e4.10g} was established in
Proposition~\ref{Pastur for covarians}(ii). Appealing then to
Theorem~\ref{2-dim exp III}, the existence of a sequence
$(\beta_j(f,g))_{j\in\N_0}$ satisfying the second equality will follow, if we
establish that the function
\[
\Delta f(x,y)=
\begin{cases}
\frac{f(x)-f(y)}{x-y}, &\mbox{if $x\ne y$,}\\
f'(x), &\mbox{if $x=y$}
\end{cases}
\]
belongs to $C^{\infty}_b(\R^2)$ for any function $f$ from
$C_b^{\infty}(\R)$. But this follows from the formula
\[
\Delta f(x,y)=\int_0^1f'(sx+(1-s)y)\6s, \qquad((x,y)\in\R^2),
\]
which together with the usual theorem on differentiation under the
integral sign shows that $\Delta f$ is a $C^{\infty}$-function on
$\R^2$ with derivatives given by
\[
\frac{\partial^{k+l}}{\partial x^k\partial y^l}\Delta f(x,y)
=\int_0^1f^{(k+l+1)}(sx+(1-s)y)s^k(1-s)^l\6s, \qquad((x,y)\in\R^2)
\]
for any $k,l$ in $\N_0$.
\end{proof}

We close this section by giving a short proof of the previously
mentioned fact that the measures $\rho_n(x,y)\6x\6y$ converge weakly
to the measure $\rho(x,y)\6x\6y$ given by \eqref{e4.3a}. As indicated
at the end of the introduction, this fact is well-known in the physics
literature (see \cite{kkp} and references therein).

\begin{proposition}\label{svag konvergens af rho-n}
For each $n$ in $\N$, let $\mu_n$ denote the measure on $\R^2$ with
density $\rho_n$ with respect to Lebesgue measure on $\R^2$.
Then $\mu_n$ is a probability measure on $\R^2$, and
$\mu_n$ converges weakly, as $n\to\infty$, to the probability measure
$\mu$ on $\R^2$ with density
\[
\rho(x,y)=\frac{1}{4\pi^2}\frac{4-xy}{\sqrt{4-x^2}\sqrt{4-y^2}}
1_{(-2,2)}(x)1_{(-2,2)}(y),
\]
with respect to Lebesgue measure on $\R^2$.
\end{proposition}

\begin{proof} We prove that
\begin{equation}
\lim_{n\to\infty}\int_{\R^2}\e^{\ri zx+\ri wy}\rho_n(x,y)\6x\6y=
\int_{\R^2}\e^{\ri zx+\ri wy}\rho(x,y)\6x\6y
\label{a4.8}
\end{equation}
for all $z,w$ in $\R$. Given such $z$ and $w$, we apply formulas
\eqref{a4.1}-\eqref{a4.4} to the case where $f(x,y)=\e^{\ri zx+\ri wy}$,
and it follows that
\begin{equation}
\begin{split}
&\int_{\R^2}\e^{\ri zx+\ri wy}\rho_n(x,y)\6x\6y\\[.2cm]
&=\frac{1}{4}\int_{\R^2}\e^{\ri zx+\ri wy}
\big[\tilde{h}_n(x)\tilde{h}_n(y)-4h_n'(x)h_n'(y)
-n^{-2}h_n''(x)h_n''(y)\big]\6x\6y\\[.2cm]
&=\frac{1}{4}\int_{\R^2}\big[4 + 2\ri zx + 2\ri wy - zwxy + 4zw
-n^{-2}z^2w^2\big]\e^{\ri zx+\ri wy}h_n(x)h_n(y)\6x\6y\\[.2cm]
&=\frac{1}{4}\int_{\R^2}\big[(4+4zw-n^{-2}z^2w^2) + 2\ri zx + 2\ri wy -
zwxy\big]\e^{\ri zx+\ri wy}h_n(x)h_n(y)\6x\6y.
\end{split}
\label{a4.5}
\end{equation}
In the case $z=w=0$, it follows in particular that $\mu_n$ is indeed a
probability measure, and hence, once \eqref{a4.8} has been
established, so is $\mu$.
By linearity the resulting expression in \eqref{a4.5} may be written
as a linear combination of 4 integrals of tensor products (a function
of $x$ times a function of $y$). Therefore, by Fubini's Theorem and
Wigner's semi-circle law, it follows
that 
\begin{equation*}\begin{split}
\lim_{n\to\infty}\int_{\R^2}\e^{\ri zx+\ri wy}&\rho_n(x,y)\6x\6y\\[.2cm] 
&=\frac{1}{4}\int_{\R^2}\big[4+4zw + 2\ri zx + 2\ri wy - zwxy\big]
\e^{\ri zx+\ri wy}h_{\infty}(x)h_{\infty}(y)\6x\6y,
\end{split}\end{equation*}
where $h_{\infty}(x)=\frac{1}{2\pi}\sqrt{4-x^2}1_{[-2,2]}(x)$. For $x$
in $(-2,2)$ it is easily seen that
\begin{equation}
h_{\infty}'(x)=\frac{-x}{2\pi\sqrt{4-x^2}}, \qand
\tilde{h}_{\infty}(x):=h_{\infty}(x)-xh'_{\infty}(x)
=\frac{2}{\pi\sqrt{4-x^2}},
\label{a4.5a}
\end{equation}
so in particular $h_{\infty}'$ and $\tilde{h}_{\infty}$ are both
$\CL^1$-functions (with respect to Lebesgue measure).
This enables us to perform the calculations in \eqref{a4.5} in the
reversed order and with $h_n$ replaced by $h_{\infty}$. We may
thus deduce that
\begin{equation}
\lim_{n\to\infty}\int_{\R^2}\e^{\ri zx+\ri wy}\rho_n(x,y)\6x\6y
=\int_{(-2,2)\times(-2,2)}\e^{\ri zx+\ri wy}
\big[\tfrac{1}{4}\tilde{h}_{\infty}(x)\tilde{h}_{\infty}(y)
-h_{\infty}'(x)h_{\infty}'(y)\big]\6x\6y.
\label{a4.6}
\end{equation}
Finally it follows from \eqref{a4.5a} and a 
straightforward calculation that
\begin{equation}
\frac{1}{4}\tilde{h}_{\infty}(x)\tilde{h}_{\infty}(y)
-h_{\infty}'(x)h_{\infty}'(y) 
=\frac{4-xy}{4\pi^2\sqrt{4-x^2}\sqrt{4-y^2}}
\label{a4.7}
\end{equation}
for all $x,y$ in $(-2,2)$. Combining \eqref{a4.6} with \eqref{a4.7},
we have established \eqref{a4.8}.
\end{proof}

\section{Asymptotic expansion for the two-dimensional Cauchy
  transform}\label{sec to-dim G} 

In this section we study in greater detail the asymptotic expansion
from Corollary~\ref{2-dim exp IV} in the case where
$f(x)=\frac{1}{\lambda-x}$ and $g(x)=\frac{1}{\mu-x}$ for
$\lambda,\mu$ in $\C\setminus\R$. In this setup we put
\[
G_n(\lambda,\mu)=\Cov\big\{\Tr_n[(\lambda\unit-X_n)^{-1}],
\Tr_n[(\mu\unit-X_n)^{-1}]\big\},
\]
where as before $X_n$ is a $\GUE(n,\frac{1}{n})$ random matrix.

Recall from Proposition~\ref{svag konvergens af rho-n} that
$\lim_{n\to\infty}G_n(\lambda,\mu)=G(\lambda,\mu)$ for any
$\lambda,\mu$ in $\C\setminus\R$, where 
\begin{equation}
G(\lambda,\mu)=\int_{\R^2}\Big(\frac{(\lambda-x)^{-1}
-(\lambda-y)^{-1}}{x-y}\Big)\Big(\frac{(\mu-x)^{-1}
-(\mu-y)^{-1}}{x-y}\Big)\rho(x,y)\6x\6y.
\label{e4.14}
\end{equation}

\begin{lemma}\label{ikke-lineaer diff ligning}
Let $G_n$ be the Cauchy transform of a $\GUE(n,\frac{1}{n})$ random
matrix. Then for any $\lambda$ in $\C\setminus\R$ we have that
\[
\tilde{G}_n(\lambda)^2-4G_n'(\lambda)^2+4G_n'(\lambda)-
\tfrac{1}{n^2}G_n''(\lambda)^2=0,
\]
where $\tilde{G}_n(\lambda)=G_n(\lambda)-\lambda G_n'(\lambda)$.
\end{lemma}

\begin{proof} For $\lambda$ in $\C\setminus\R$ we put
\[
K_n(\lambda)=\tilde{G}_n(\lambda)^2-4G_n'(\lambda)^2+4G_n'(\lambda)-
\tfrac{1}{n^2}G_n''(\lambda)^2.
\]
Observing that $\tilde{G}_n'(\lambda)=-\lambda G_n''(\lambda)$, it
follows that for any $\lambda$ in $\C\setminus\R$
\begin{equation*}\begin{split}
K_n'(\lambda)&=2\tilde{G}_n(\lambda)\tilde{G}_n'(\lambda)
-8G_n'(\lambda)G_n''(\lambda)+4G_n''(\lambda)-
\tfrac{2}{n^2}G_n''(\lambda)G_n'''(\lambda)\\[.2cm]
&=2G_n''(\lambda)\big[-\lambda G_n(\lambda)+\lambda^2G_n'(\lambda)
-4G_n'(\lambda)+2-\tfrac{1}{n^2}G_n'''(\lambda)\big]\\[.2cm]
&=2G_n''(\lambda)\big[-\tfrac{1}{n^2}G_n'''(\lambda)
+(\lambda^2-4)G_n'(\lambda)-\lambda G_n(\lambda)+2\big]\\[.2cm]
&=0,
\end{split}\end{equation*}
where the last equality follows from Lemma~\ref{diff ligning for
  G-n}. We may thus conclude that $K_n$ is constant on each of the two
  connected components of $\C\setminus\R$. However, for $y$ in $\R$ we
  have by dominated convergence that
\[
|\ri yG_n'(\ri y)|=\Big|y\int_{\R}\frac{1}{(\ri y-x)^2}h_n(x)\6x\Big|
\le\int_{\R}\frac{|y|}{y^2+x^2}h_n(x)\6x\longrightarrow0,
\]
as $|y|\to\infty$, and similarly $G_n(\ri y)\to0$
and $G_n''(\ri y)\to0$ as $|y|\to\infty$. It thus follows that
$K_n(\ri y)\to0$, as $|y|\to\infty$, $y\in\R$, and hence we must have
$K_n\equiv0$, as desired.
\end{proof}

\begin{theorem}\label{2-dim G vha 1-dim}
Let $X_n$ be a $\GUE(n,\frac{1}{n})$ random matrix, and consider for
$\lambda,\mu$ in $\C\setminus\R$ the two-dimensional Cauchy transform:
\[
G_n(\lambda,\mu)=\Cov\big\{\Tr_n[(\lambda\unit-X_n)^{-1}],
\Tr_n[(\mu\unit-X_n)^{-1}]\big\}.
\]

\begin{itemize}

\item[(i)] If $\lambda\ne\mu$, we have that
\[
G_n(\lambda,\mu)=\frac{-1}{2(\lambda-\mu)^2}
\big(\tilde{G}_n(\lambda)\tilde{G}_n(\mu)-(2G_n'(\lambda)-1)(2G_n'(\mu)-1)+1
-\tfrac{1}{n^2}G_n''(\lambda)G_n''(\mu)\big),
\]
where $G_n(\lambda)$ is the Cauchy transform of $X_n$ at $\lambda$, and
where $\tilde{G}_n(\lambda)=G_n(\lambda)-\lambda G_n'(\lambda)$.

\item[(ii)] If $\lambda=\mu\in\C\setminus\R$ we have that
\[
\V\big\{\Tr_n[(\lambda\unit-X_n)^{-1}]\big\}
=G_n(\lambda,\lambda)
=\frac{1}{4}(\lambda^2-4)G_n''(\lambda)^2-\frac{1}{4n^{2}}G_n'''(\lambda)^2
\]
with $G_n(\lambda)$ as in (i).
\end{itemize}
\end{theorem}

\begin{proof} (i) \ Assume that $\lambda,\mu\in\C\setminus\R$, and that
$\lambda\ne\mu$. Using Corollary~\ref{Pastur for covarians}(ii) we
find that 
\begin{equation*}\begin{split}
G_n(\lambda,\mu)&=\int_{\R^2}\Big(\frac{(\lambda-x)^{-1}
-(\lambda-y)^{-1}}{x-y}\Big)\Big(\frac{(\mu-x)^{-1}
-(\mu-y)^{-1}}{x-y}\Big)\rho_n(x,y)\6x\6y\\[.2cm]
&=\int_{\R^2}\frac{1}{(\lambda-x)(\mu-x)(\lambda-y)(\mu-y)}
\rho_n(x,y)\6x\6x\\[.2cm]
&=\frac{1}{(\mu-\lambda)^2}
\int_{\R^2}\Big(\frac{1}{\lambda-x}-\frac{1}{\mu-x}\Big)
\Big(\frac{1}{\lambda-y}-\frac{1}{\mu-y}\Big)\rho_n(x,y)
\6x\6y.
\end{split}\end{equation*}
Using now Proposition~\ref{formel for rho-n} and Fubini's Theorem, it
follows that
\begin{equation}
G_n(\lambda,\mu)=\frac{1}{4(\mu-\lambda)^2}
\big((H_n(\lambda)-H_n(\mu))^2-4(I_n(\lambda)-I_n(\mu))^2-
\frac{1}{n^2}(J_n(\lambda)-J_n(\mu))^2\big),
\label{e4.12}
\end{equation}
where e.g.\
\[
H_n(\lambda)=\int_{\R}\frac{1}{\lambda-x}\tilde{h}_n(x)\6x,\
I_n(\lambda)=\int_{\R}\frac{1}{\lambda-x}h_n'(x)\6x, \mbox{ and }
J_n(\lambda)=\int_{\R}\frac{1}{\lambda-x}h_n''(x)\6x.
\]
Note here that by partial integration and \eqref{e3.15}
\begin{equation*}\begin{split}
\int_{\R}\frac{1}{\lambda-x}\tilde{h}_n(x)\6x
&=\int_{\R}\frac{1}{\lambda-x}(h_n(x)-xh_n'(x))\6x
=G_n(\lambda)-\int_{\R}\Big(\frac{\lambda}{\lambda-x}-1\Big)
h_n'(x)\6x\\[.2cm]
&=G_n(\lambda)+\lambda\int_{\R}\frac{1}{(\lambda-x)^2}h_n(x)\6x
=G_n(\lambda)-\lambda G_n'(\lambda)\\[.2cm]
&=\tilde{G}_n(\lambda).
\end{split}\end{equation*}
We find similarly that
\[
\int_{\R}\frac{1}{\lambda-x}h_n'(x)\6x = G_n'(\lambda),
\qand
\int_{\R}\frac{1}{\lambda-x}h_n''(x)\6x = G_n''(\lambda).
\]
Inserting these expressions into \eqref{e4.12}, it follows that
\begin{equation*}\begin{split}
4(\lambda-\mu)^2G(\lambda,\mu)&=
\big(\tilde{G}_n(\lambda)-\tilde{G}_n(\mu)\big)^2 
-4\big(G_n'(\lambda)-G_n'(\mu)\big)^2
-\tfrac{1}{n^2}\big(G_n''(\lambda)-G_n''(\mu)\big)^2
\\[.2cm]
&=\Big[\tilde{G}_n(\lambda)^2-4G_n'(\lambda)^2
-\tfrac{1}{n^2}G_n''(\lambda)^2\Big]
+\Big[\tilde{G}_n(\mu)^2-4G_n'(\mu)^2
-\tfrac{1}{n^2}G_n''(\mu)^2\Big]\\
&\qquad\qquad -2\tilde{G}_n(\lambda)\tilde{G}_n(\mu)
+8G_n'(\lambda)G_n'(\mu)+\tfrac{2}{n^2}G_n''(\lambda)G_n''(\mu)
\\[.2cm]
&=-4G_n'(\lambda)-4G_n'(\mu)-2\tilde{G}_n(\lambda)\tilde{G}_n(\mu)
+8G_n'(\lambda)G_n'(\mu)+\tfrac{2}{n^2}G_n''(\lambda)G_n''(\mu),
\end{split}\end{equation*}
where the last equality uses Lemma~\ref{ikke-lineaer diff ligning}. We
may thus conclude that
\begin{equation*}\begin{split}
G_n(\lambda,\mu)&=
\frac{-1}{2(\lambda-\mu)^2}
\Big(\tilde{G}_n(\lambda)\tilde{G}_n(\mu)+2G_n'(\lambda)+2G_n'(\mu)
-4G_n'(\lambda)G_n'(\mu)-\tfrac{1}{n^2}G_n''(\lambda)G_n''(\mu)\Big)
\\[.2cm]
&=\frac{-1}{2(\lambda-\mu)^2}
\Big(\tilde{G}_n(\lambda)\tilde{G}_n(\mu)-(2G_n'(\lambda)-1)(2G_n'(\mu)-1)+1
-\tfrac{1}{n^2}G_n''(\lambda)G_n''(\mu)\Big),
\end{split}\end{equation*}
which completes the proof of (i).

(ii) \ Proceeding as in the proof of (i) we find by application of
Proposition~\ref{formel for rho-n} that
\begin{equation}
\begin{split}
4G_n(\lambda,\lambda)&=4\int_{\R}\frac{1}{(\lambda-x)^2(\lambda-y)^2}
\rho_n(x,y)\6x\6y\\[.2cm]
&=\Big(\int_{\R}\frac{\tilde{h}_n(x)}{(\lambda-x)^2}\6x\Big)^2
-4\Big(\int_{\R}\frac{h'_n(x)}{(\lambda-x)^2}\6x\Big)^2
-\frac{1}{n^2}\Big(\int_{\R}\frac{h''_n(x)}{(\lambda-x)^2}\6x\Big)^2.
\end{split}
\label{e4.13}
\end{equation}
By calculations similar to those in the proof of (i), we have here
that
\[
\int_{\R}\frac{\tilde{h}_n(x)}{(\lambda-x)^2}\6x=\lambda G_n''(\lambda),
\quad \int_{\R}\frac{h'_n(x)}{(\lambda-x)^2}\6x=-G_n''(\lambda), \quad
\int_{\R}\frac{h''_n(x)}{(\lambda-x)^2}\6x=-G_n'''(\lambda),
\]
which inserted into \eqref{e4.13} yields the formula in (ii).
\end{proof}

\begin{corollary}\label{asymp exp for 2-dim G}
Consider the coefficients $\eta_j$, $j\in\N_0$, in the asymptotic
expansion of $G_n(\lambda)$ (cf.\ Proposition~\ref{eksplicit udtryk
  for eta}), and adopt as before the notation
$\tilde{\eta}_j(\lambda)=\eta_j(\lambda)-\lambda\eta_j'(\lambda)$.

\begin{itemize}

\item[(i)]
For any distinct $\lambda,\mu$ from
$\C\setminus\{0\}$ and $k$ in $\N_0$ we have the asymptotic expansion:
\begin{equation}
G_n(\lambda,\mu)=
\frac{1}{2(\lambda-\mu)^2}\Big[\Gamma_0(\lambda,\mu)+
\frac{\Gamma_1(\lambda,\mu)}{n^2}+\frac{\Gamma_2(\lambda,\mu)}{n^4}+\cdots+
\frac{\Gamma_k(\lambda,\mu)}{n^{2k}}+O(n^{-2k-2})\Big],
\label{e4.23}
\end{equation}
where
\[
\Gamma_0(\lambda,\mu)=(2\eta_0'(\lambda)-1)(2\eta_0'(\mu)-1)
-\tilde{\eta}_0(\lambda)\tilde{\eta}_0(\mu)-1,
\]
and for $l$ in $\{1,2,\ldots,k\}$
\begin{equation}
\begin{split}
\Gamma_l(\lambda,\mu)&=
2\eta_l'(\lambda)(2\eta_0'(\mu)-1)+2\eta_l'(\mu)(2\eta_0'(\lambda)-1)
+4\sum_{j=1}^{l-1}\eta_j'(\lambda)\eta_{l-j}'(\mu)\\
&\phantom{\Gamma_l(\lambda,\mu)=}\quad
+\sum_{j=0}^{l-1}\eta_j''(\lambda)\eta_{l-1-j}''(\mu)
-\sum_{j=0}^l\tilde{\eta}_j(\lambda)\tilde{\eta}_{l-j}(\mu)
\end{split}
\label{e4.16}
\end{equation}
(the third term on the right hand side should be neglected, when
$l=1$). 

\item[(ii)] For any $\lambda$ in $\C\setminus\R$ and any $k$ in $\N_0$
  we have that 
\[
G_n(\lambda,\lambda)=\frac{1}{4}\Big[\Upsilon_0(\lambda)
+\frac{\Upsilon_1(\lambda)}{n^2}+\frac{\Upsilon_2(\lambda)}{n^4} 
+\cdots+\frac{\Upsilon_k(\lambda)}{n^{2k}}+O(n^{-2k-2})\Big],
\]
where
\[
\Upsilon_0(\lambda)=(\lambda^2-4)\eta_0''(\lambda)^2,
\]
and for $l$ in $\{1,2,\ldots,k\}$
\begin{equation}
\Upsilon_l(\lambda)=
(\lambda^2-4)\sum_{j=0}^l\eta_j''(\lambda)\eta_{l-j}''(\lambda)
-\sum_{j=0}^{l-1}\eta_j'''(\lambda)\eta_{l-1-j}'''(\lambda).
\label{e4.17a}
\end{equation}
\end{itemize}
\end{corollary}

\begin{proof} From the asymptotic expansion of $G_n(\lambda)$ 
(cf.\ Proposition~\ref{eksplicit udtryk for eta}) it follows that
\begin{equation*}\begin{split}
2G_n'(\lambda)-1&=(2\eta_0'(\lambda)-1)+\frac{2\eta_1'(\lambda)}{n^2}+\cdots+
\frac{2\eta_k'(\lambda)}{n^{2k}}+O(n^{-2k-2})\\[.2cm]
G_n''(\lambda)&=\eta_0''(\lambda)+\frac{\eta_1''(\lambda)}{n^2}+\cdots+
\frac{\eta_k''(\lambda)}{n^{2k}}+O(n^{-2k-2})\\[.2cm]
\tilde{G}_n(\lambda)&=\tilde{\eta}_0(\lambda)
+\frac{\tilde{\eta}_1(\lambda)}{n^2}+\cdots+
\frac{\tilde{\eta}_k(\lambda)}{n^{2k}}+O(n^{-2k-2}),
\end{split}\end{equation*}
where we also use that the derivatives of the remainder terms are
controlled via Lemma~\ref{analytisk}.

Inserting the above expressions (and the corresponding expressions for
$2G_n'(\mu)-1$, $G_n''(\mu)$ and $\tilde{G}_n(\mu)$)
into the formula in Theorem~\ref{2-dim G vha 1-dim}(i), it is
straightforward to establish (i) by collecting the
$n^{-2l}$-terms for each $l$ in $\{0,1,\ldots,k\}$. The proof of (ii)
follows similarly from Theorem~\ref{2-dim G vha 1-dim}(ii).
\end{proof}

\begin{remark}
Using that $\eta_0(\lambda)=\frac{\lambda}{2}-\frac{1}{2}(\lambda^2-4)^{1/2}$
(cf.\ Proposition~\ref{eksplicit udtryk for eta}) it follows from
Corollary~\ref{asymp exp for 2-dim G}(i) and a straightforward
calculation that for distinct $\lambda$ and $\mu$ from $\C\setminus\R$,
\begin{equation}
G(\lambda,\mu)=\lim_{n\to\infty}G_n(\lambda,\mu)=
\frac{\Gamma_0(\lambda,\mu)}{2(\lambda-\mu)^2}
=\frac{1}{2(\lambda-\mu)^2}\Big(\frac{\lambda\mu-4}
{(\lambda^2-4)^{1/2}(\mu^2-4)^{1/2}}-1\Big), 
\label{e4.15}
\end{equation}
where $G(\lambda,\mu)$ was initially given by \eqref{e4.14}. If
$\lambda=\mu$, it follows similarly from Corollary~\ref{asymp exp for
  2-dim G}(ii) that
\[
G(\lambda,\lambda)=\lim_{n\to\infty}G_n(\lambda,\lambda)
=\frac{1}{4}(\lambda^2-4)\eta_0''(\lambda)^2
=\frac{1}{(\lambda^2-4)^2},
\]
which may also be obtained by letting $\lambda$ tend to $\mu$
in \eqref{e4.15}!

Using also that $\eta_1(\lambda)=(\lambda^2-4)^{-5/2}$, it follows
from \eqref{e4.16} and a rather tedious calculation that
\begin{equation*}\begin{split}
\Gamma_1(\lambda,\mu)=
\frac{(\lambda-\mu)^2}{(\lambda^2-4)^{7/2}(\mu^2-4)^{7/2}}
\big(&5\lambda\mu^5+4\lambda^2\mu^4+4\mu^4-52\lambda\mu^3+
3\lambda^3\mu^3-16\mu^2+4\lambda^4\mu^2\\[.2cm] 
&-52\lambda^2\mu^2+208\lambda\mu+5\lambda^5\mu-52\lambda^3\mu-16\lambda^2+
320+4\lambda^4\big).
\end{split}\end{equation*}
Inserting this into \eqref{e4.23} and letting $\lambda$ tend to $\mu$
we obtain that
\[
\Upsilon_1(\lambda)=4(21\lambda^2+20)(\lambda^2-4)^{-5},
\]
which is in accordance with \eqref{e4.17a}.
\end{remark}

\vspace{1.5cm}

\begin{minipage}[c]{0.5\textwidth}
Department of Mathematical Sciences\\
University of Copenhagen\\
Universitetsparken~5\\
2100 Copenhagen {\O}\\
Denmark\\
{\tt haagerup@math.ku.dk}
\end{minipage}
\hfill
\begin{minipage}[c]{0.5\textwidth}
Department of Mathematical Sciences\\
University of Aarhus\\
Ny Munkegade 118\\
8000 Aarhus C\\
Denmark\\
{\tt steenth@imf.au.dk}
\end{minipage}


{\small
\begin{thebibliography}{999999}

\bibitem[APS]{aps} {\sc S.~Albeverio, L.~Pastur and M.~Shcherbina},
  {\it On the $1/n$ Expansion for Some Unitary Invariant Ensembles of
    Random Matrices}, Commun.~Math.~Physics {\bf 224} (2001), 271-305.

\bibitem[C-D]{c-d} {\sc T.~Cabanal-Duvillard}, {\it Fluctuations de la
    loi empirique de grandes matrices al\'eatoires},
  Ann.~I.~H.~Poincar\'e - PR37, {\bf 3} (2001), 373-402.

\bibitem[EM]{em} {\sc N.M.~Ercolani and K.D.T.-R.~Mclaughlin}, {\it
    Asymptotics of the partition function for random matrices via
    Riemann-Hilbert techniques and applications to graphical
    enumeration}, International Mathematics Research Notes {\bf 14}
    (2003), 755-820.

\bibitem[GT1]{gt} {\sc F.~G\"otze and A.~Tikhomirov}, {\it The rate of
    convergence for spectra of GUE and LUE matrix ensembles},
    Cent.~Eur.~J.~Math.\ {\bf 3} (2005), no.~4, 666--704. 

\bibitem[GT2]{gt2} {\sc F.~G\"otze and A.~Tikhomirov}, {\it Limit
    theorems for spectra of random matrices with martingale
    structure}, Theory Probab.~Appl.\ {\bf 51}, 42-64.

\bibitem[HST]{hst}
{\sc U.~Haagerup, H.~Schultz} and {\sc S.~Thobj{\o}rnsen}, {\sl A
    random matrix proof of the lack of projections in
    $C^*_{\rm red}(F_2)$}, Adv.\ Math.\ {\bf 204} (2006), 1-83.

\bibitem[HT1]{ht1} 
{\sc U.~Haagerup and S.~Thorbj\o{}rnsen}, {\sl Random
  Matrices with Complex Gaussian Entries}, Expositiones Math.\ {\bf 21}
  (2003), 293-337.

\bibitem[HT2]{ht5}
{\sc U.~Haagerup and S.~Thobj{\o}rnsen},
{\it A new application of random matrices:
$\Ext(\Cred(F_2))$ is not a group}, Annals of Math.\ {\bf 162} (2005),
711--775. 

\bibitem[HTF]{htf} {\it Higher Transcendental Functions vol.~1-3}, 
{\sc A.~Erd\'elyi, W.~Magnus, F.~Oberhettinger and
F.G.~Tricomi} (editors), based in part on notes left by {\sc H.~Bateman}, 
McGraw-Hill Book Company Inc.\ (1953-55).

\bibitem[HZ]{hz} {\sc J.~Harer and D.~Zagier}, {\sl The Euler characteristic
of the modulo space of curves}, Invent.\ Math.\ {\bf 85} (1986),
457-485. 

\bibitem[KKP]{kkp} {\sc A.~Khorunzhy, B.~Khoruzhenko and L~Pastur},
{\it Asymptotic properties of large random matrices with independent
  entries}, J.~Math.~Phys.\ {\bf 37}(10) (1996), 5033-5060.
 
\bibitem[MN]{mn} {\sc J.~A.~Mingo and A.~Nica}
{\it Annular noncrossing permutations and partitions, and second-order
  asymptotics for random matrices.} Int.\ Math.\ Res.\ Not.\  {\bf 28}
(2004), 1413--1460.  

\bibitem[PS]{ps} {\sc L.~Pastur and M.~Scherbina}, {\it Universality
    of the local eigenvalue statistics for a class of unitary
    invariant random matrix ensembles}, J.~Stat.\ Physics, {\bf 86}
    (1997), 109-147.

\bibitem[Ru]{ru} {\sc W.~Rudin}, {\it Functional Analysis (second
    edition)}, McGraw-Hill (1991). 

\bibitem[Sc]{sc} {\sc R.L.~Schilling}, {\it Measures, Integrals and
    Martingales}, Cambridge University Press (2005).

\bibitem[Th]{th} {\sc S.~Thorbj{\o}rnsen}, {\it Mixed Moments of
Voiculescu's Gaussian Random Matrices}, J.~Funct.~Anal.\ {\bf 176}
(2000), 213-246.

\bibitem[Vo]{vo} {\sc D.~Voiculescu}, Limit laws for random matrices and free
products, {\it Inventiones Math.} 104 (1991), 202--220.

\bibitem[VDN]{vdn} {\sc D.~Voiculescu, K.~Dykema} and {\sc A.~Nica},
  Free Random Variables, {\it CMR Monograph Series} 1, {\it American
  Mathematical Society} (1992). 

\end{thebibliography}}
\end{document}